\documentclass[10pt]{article}

\usepackage{a4wide,graphicx}
\usepackage{amsmath,amssymb} 
\usepackage{amsthm} 
\usepackage{color}
\usepackage{subfig}
\usepackage{caption}
\usepackage{comment}
\usepackage{enumerate}
\usepackage[left=1in, right=1in]{geometry}
\usepackage{bbm}

\usepackage[color]{changebar}
\usepackage[normalem]{ulem}

\let\pa\partial  
  
\let\eps\varepsilon  
\newcommand{\N}{{\mathbb N}} 
\newcommand{\R}{{\mathbb R}}

\newcommand{\dist}{\operatorname{dist}}  

\newcommand{\T}{{{\mathbb T}^2}}

\newcommand{\A}{{\mathbb A}}
\newcommand{\dd}{{\mathrm{d}}}
\newcommand{\sym}{\operatorname{sym}}

\newcommand{\tensor}[1]{\mathbb #1}

\newcommand{\mat}[1]{\boldsymbol #1}
\newcommand{\vect}[1]{\boldsymbol #1}
\newcommand{\energy}{\mathcal{J}}
\newcommand{\fenergy}{\widetilde{\mathcal{J}}}
\newcommand{\Khp}{\mathcal{K}_{(h_n)}^+}
\newcommand{\Khpm}{\mathcal{K}_{(h_n)}^\pm}
\newcommand{\Khm}{\mathcal{K}_{(h_n)}^-}
\newcommand{\Kchi}{\mathcal{K}_{(\chi^{h_n})}}
\newcommand{\Kchim}{\mathcal{K}_{(\chi^{h_n})}^-}
\newcommand{\Kchip}{\mathcal{K}_{(\chi^{h_n})}^+}
\newcommand{\Kh}{\mathcal{K}_{(h_n)}}
\newcommand{\set}[1]{\mathcal #1}

\newcommand{\Khtp}{\widetilde{\mathcal{K}}_{(h_n)}^+}
\newcommand{\Khtpm}{\widetilde{\mathcal{K}}_{(h_n)}^\pm}
\newcommand{\Khtm}{\widetilde{\mathcal{K}}_{(h_n)}^-}
\newcommand{\Kht}{\widetilde{\mathcal{K}}_{(h_n)}}
\newcommand{\Kchit}{\widetilde{\mathcal{K}}_{(\chi^{h_n})}}

\newcommand{\SO}[1]{\textrm{SO} (#1)}
\newcommand{\inte}{\operatorname{int}}
\newcommand{\diam}{\operatorname{diam}}

\newtheorem{theorem}{Theorem}[section] 
\newtheorem{lemma}[theorem]{Lemma}   
\newtheorem{proposition}[theorem]{Proposition}  
\newtheorem{corollary}[theorem]{Corollary}  

\newtheorem{definition}{Definition}[section] 
\newtheorem{example}{Example}[section] 
\newtheorem{assumption}{Assumption}[section] 

\newtheorem{remark}{Remark}[section]

\begin{document}


\title{On the simultaneous homogenization and dimension reduction in 
elasticity and locality of $\Gamma$-closure}

\author{Mario Bukal\thanks{University of Zagreb, Faculty of Electrical Engineering 
and Computing, Unska 3, 10000 Zagreb, Croatia, \newline
E-Mails: mario.bukal@fer.hr, igor.velcic@fer.hr}, 
 and Igor Vel\v ci\'c\footnotemark[1]}

%
%
%

\maketitle 

\begin{abstract}
On the example of linearized elasticity we provide a framework for simultaneous homogenization and dimension reduction in the setting of linearized elasticity as well as non-linear elasticity  for the derivation of homogenized von K\'arm\'an plate and bending rod models. The framework encompasses even perforated domains and domains with oscillatory boundary, provided that the corresponding extension operator can be constructed.
Locality property of $\Gamma$-closure is established, i.e.~every energy density 
obtained by the homogenization process can be in almost every point be obtained as the limit of periodic energy densities. 
\end{abstract}

\vspace{4mm}

\vspace{2mm}

\section{Introduction and main results}
Our starting point is the three-dimensional linearized elasticity framework \cite{Cia97}, where
the stored elasticity energy of a material is given by a quadratic form
\begin{equation}\label{intro:def.ee}
\frac12\int_{\hat{\Omega}}\tensor{A}(\hat x)\sym\nabla\vect{u}:\sym\nabla\vect{u}\,\dd\hat x\,.
\end{equation}
Above $\hat\Omega\subset\R^3$ describes a reference configuration of material, $\vect{u}:\hat{\Omega}\to \R^3$ is
the displacement field, $\sym\nabla\vect{u} = (\nabla\vect{u} + {\nabla\vect{u}}^{\tau})/2$ symmetrized gradient, 
and $\tensor{A}$ is the elasticity tensor describing material properties.

In this paper we consider composite platelike materials with the aim of studying the asymptotic behaviour of a sequence of energies (\ref{intro:def.ee}), 
parametrized by the vanishing body thickness, and deriving homogenized linear plate model by means of
simultaneous homogenization and dimension reduction. 
Such a problem has been already discussed in 
\cite{DamVog87}, where the authors derived the model of homogenized plate on the level of 
linearized elasticity system of equations using compensated compactness argument (see \cite{MuTa97}) 
and assuming that external loads act in the vertical direction. In that way they obtained a limit model, which is purely (linear) bending model.
We also mention the work \cite{GuMo06}, where in the context of linearized elasticity 
the authors derived, again by using the compensated compactness, the model of homogenized 
plate for elastic laminates (layered materials). 
The limit problem is realized in several different ways with explicitely given elasticity tensors.
In contrast to these, our approach is completely variational, resembling
the classical $\Gamma$-convergence method \cite{Bra02,DMa93}, and relies on techniques developed by the second author in the framework 
of deriving homogenized bending rod \cite{MaVe14a} 
and von K\'arm\'an plate models \cite{Vel14a} from  three-dimensional non-linear elasticity theory.
The approach includes materials which oscillate both in in-plane and out-of-plane directions,
and even perforated materials and materials with oscillatroy boundary, provided that certain extension 
operator can be constructed. Moreover, we consider the full model in the sense that it admits external 
loads in all directions with appropriate scaling.

To the best of our knowledge, the problem of simultaneous homogenization and dimension reduction
solved in this paper, even in the context of the linearized elasticity, 
cannot be put into any existing abstract framework. The reason for that lies precisely in the 
simultaneity of the approach (cf.~\cite{Cia97}).
In the first part of the paper we utilize the arguments 
given in \cite{Vel14a}, simplified to the case of the linearized elasticity, and demonstrate 
how can we elegantly derive the model of linear plate from the $3D$ linearized elasticity. 
Simplification (connected with a bit change)  of the arguments presented in \cite{MaVe14a,Vel14a}  
(see the proof of Lemma \ref{lemma:equi-int1} below),
along with the possibility of analyzing domains (materials) with holes and oscillating boundary
on the abstract level,
are the main contributions of the first part. 
We emphasize that our approach consists
in defining the correctors on the energetic level  (see Lemma
\ref{lemma:equi-int1} below), although, the approach can be used for the analysis of equations as well 
(see \cite{buk16}). We also point out that the method presented here, as well as 
in \cite{MaVe14a,Vel14a}, is not only limited to dimension reduction problems in elasticity.

Performing the rescaling of the reference configuration $\Omega_h = \omega\times[-\frac{h}{2},\frac{h}{2}]$,
where $\omega\subset\R^2$ describes the shape and $h>0$ thickness of the plate, to the fixed domain
$\Omega = \omega\times I$ with $I:= [-\frac{1}{2},\frac{1}{2}]$,
and dividing the elastic energy by the order of volume $h$, expression 
(\ref{intro:def.ee}) amounts to
\begin{equation*}\label{intro:def.eeh}
\energy_h(\vect u) = \frac12\int_\Omega \tensor{A}^h(x)\sym\nabla_h \vect{u} : 
\sym\nabla_h\vect{u}\dd x=\int_\Omega Q^h(x,\sym\nabla_h \vect{u})\dd x \,,
\end{equation*}
where $\nabla_h = (\nabla',\frac{1}{h}\pa_3)$ denotes the scaled gradient 
and $\A^h (x) := \A(h,x_1,x_2,hx_3)$ the scaled elasticity tensor ($\A(h,\cdot)$ is the elasticity tensor on the domain $\Omega_h$ and $Q^h(x,\mat F)=\frac12\tensor{A}^{h}(x)\mat F : \mat F$ is the corresponding quadratic form).
Additionally, we require that the family of   quadratic forms $(Q^h)_{h>0}$
satisfies the uniform boundedness and coercivity estimates on symmetric matrices and that they assign 
zero value to skew symmetric matrices. Denoting by 
$Q^h(x,\mat F)=\frac12\tensor{A}^{h}(x)\mat F : \mat F$ the corresponding quadratic form,
then there exist positive constants $0 < \alpha \leq \beta$, independent of $h>0$, such that:   
\begin{eqnarray}\label{2.eq:A1}
\textrm{(coercivity)} & & Q^h(x,\mat F) \geq \alpha |\sym \mat{F}|^2\quad \forall\, 
\mat F\in\R^{3\times3}\,, \text{ for a.e.~} x \in\Omega\,; \\ \nonumber
\textrm{(uniform boundedness)} & & Q^h(x, \mat F) \leq \beta |\sym \mat{F}|^2\quad \forall\, 
\mat F\in\R^{3\times3}\,, \text{ for a.e.~} x \in\Omega\,.
\end{eqnarray}
Notice that from the uniform boundedness and positivity of $Q^h$ it follows 
$$
 Q^h(x,\mat F) = Q^h(x, \sym \mat F)\,, \quad \forall \mat F \in \R^{3\times 3}\,,
 \textrm{ for a.e.~} x \in \Omega\,, 
$$
and we also have 
\begin{equation}\label{razlika} 
\left| Q^h(x,\mat F_1)-Q^h(x,\mat F_2)\right| 
\leq \beta |\sym \mat F_1-\sym \mat F_2| |\sym \mat F_1+\sym \mat F_2|\,,\quad 
\forall \mat F_1, \mat F_2 \in \R^{3\times 3}, \textrm{ for a.e.~} x \in \Omega\,.
\end{equation}
                
Taking an arbitrary sequence $(h_n)_n$ of plate thickness decreasing to zero, 
we aim to describe the asymptotic behaviour of the sequence of the energy functionals $\energy_{h_n}$.
In the following we outline key waypoints in the derivation of the homogenized linear plate model.
First, let us denote by $H_{\Gamma_d}^1(\Omega,\R^3) = \{\vect u \in H^1(\Omega,\R^3)\ :\ \vect u|_{\Gamma_d\times I} = 0\}$, 
the space of displacement fields which are fixed to zero on a portion $\Gamma_d\times I$ 
($\Gamma_d\subset\omega$) of positive surface measure of the lateral boundary of $\Omega$. 
We also denote by $H_{\Gamma_d}^1(\omega,\R^2) = \{\vect u \in H^1(\omega,\R^2)\ :\ \vect u|_{\Gamma_d} = 0\}$, and 
by  $H_{\Gamma_d}^2(\omega)=\{v \in H^2(\omega)\ :\  v|_{\Gamma_d} = 0,\ \partial_{\alpha}v|_{\Gamma_d} = 0, \ \textrm{ for } \alpha=1,2 \}$.
Applying the Griso's decomposition (Lemma \ref{app:lem.limsup}) 
on a given sequence of displacement fields $(\vect{u}^{h_n})\subset H^1_{\Gamma_d}(\Omega,\R^3)$ of equi-bounded energies, we decompose its symmetrized scaled gradients in the form
\begin{equation}\label{intro:eq.ssg}
\sym\nabla_{h_n}\vect u^{h_n} = \imath(\sym\nabla'\vect w - x_3\nabla'^2 v ) + \sym\nabla_{h_n}\tilde{\vect u}^{h_n}\,,
\end{equation}
where $\vect{w}\in H^1_{\Gamma_d}(\omega,\R^2)$ and $v\in H^2_{\Gamma_d}(\omega)$ are horizontal in-plane and vertical displacements which build the fixed part, 
and $\tilde{\vect{u}}^{h_n}\in H^1_{\Gamma_d}(\Omega,\R^3)$ is a corrector which builds the relaxational part of $\sym\nabla_{h_n}\vect u^{h_n}$.
Here $\imath$ denotes the canonical embedding of $\R^{2\times2}$ into $\R^{3\times3}$, see Section \ref{sec:not} below.
Motivated by the above decomposition we define the space of matrix fields which appear as the fixed part of 
symmetrized scaled gradients
\begin{equation*}
\set{S}(\omega) = \{ \mat M_1 + x_3\mat M_2\ :\ 
\mat M_1,\mat M_2\in L^2(\omega,\R^{2\times2}_{\mathrm{sym}})\,,\ x_3\in I\}\,.
\end{equation*}
Following the standard approach  of 
$\Gamma$-convergence, for $(h_n)_{n}$ monotonically decreasing to zero, 
$A\subset\omega$ open subset and $\mat M\in \set S(\omega)$ 
define:
\begin{align*}
\Khm(\mat M,A) &= \inf\Big\{\liminf_{n\to\infty}\int_{\Omega} 
		Q^{h_n}(x,\imath(\mat M\mathbbm{1}_{A \times I}) + \nabla_{h_n}\vect\psi^{h_n})\dd x\ | \\
 		& \qquad\ ~\quad(\psi_1^{h_n},\psi_2^{h_n},{h_n}\psi_3^{h_n}) \to 0 \text{ strongly in }L^2(\Omega,\R^3)\Big\}\,;\nonumber
\end{align*}
\begin{align*}
\Khp(\mat M,A) &= \inf\Big\{\limsup_{n\to\infty}\int_{\Omega}   
		Q^{h_n}(x,\imath(\mat M\mathbbm{1}_{A \times I}) + \nabla_{h_n}\vect\psi^{h_n})\dd x\ | \\
		& \qquad\ ~\quad(\psi_1^{h_n},\psi_2^{h_n},{h_n}\psi_3^{h_n}) \to 0 \text{ strongly in }L^2(\Omega,\R^3)\Big\}. \nonumber
\end{align*}
These functionals play the role similar to lower and upper $\Gamma$-limits, respectively, and infimum in the definition 
is taken over all sequences of vector fields $(\vect{\psi}^{h_n})_{n}\subset H^1(A\times I,\R^3)$ such that
$(\psi_1^{h_n},\psi_2^{h_n},{h_n}\psi_3^{h_n}) \to 0$ strongly in the $L^2$-topology.
Establishing the equality (on a subsequence of $(h_n)_n$) between lower and upper bound
\begin{equation}
\Khm(\mat M,A) = \Khp(\mat M,A) =: \Kh(\mat M,A)\,,\quad \mat M\in \set{S}(\omega)\,,\  
A\subset\omega\,\text{open with Lipschitz boundary},
\end{equation}
and using the properties of the variational functional $\Kh$ (see also \cite[Lemma 3.7]{Vel14a}), assures an integral
representation of the variational functional (cf.~Proposition \ref{prop:Int_form}), i.e.~there exists
a function $Q^0$ (dependent on the sequence $(h_n)_n$) such that
\begin{equation}
\Kh(\mat M,A) = \int_A Q^0(x',\mat M_1(x'),\mat M_2(x'))\,\dd x'\,.
\end{equation}
Referring to Section \ref{sec:3} for details, we construct the limit energy functional (finite on Kirchoff-Love displacements, see Definition \ref{2.def:conv_displ} and Remark \ref{bukal100} )
\begin{equation}
\energy_0(\vect w, v) = \int_\omega Q^0(x',\sym\nabla'\vect w,-\nabla'^2v)\dd x'\,
\end{equation}
and provide the convergence analysis of $\energy_{h_n}(\vect u^{h_n})\to \energy_0(\vect w, v)$ as $n\to\infty$
by means of the following theorem. 
\begin{theorem}\label{thm:1}~
Let $(h_n)_n$ be monotonically decreasing to zero sequence of plate thickness that satisfies 
Assumption \ref{2:ass}.
\begin{enumerate}[(i)]
  \item \emph{(Compactness)} Let $(\vect u^{h_n})_n\subset H^1_{\Gamma_d}(\Omega,\R^3)$ be a sequence of 
  {\em equi-bounded} energies, i.e.~\\$\limsup_{n\to\infty} \energy_{h_n}(\vect{u^{h_n}}) < \infty$.
  Then there exists $(\vect w, v)\in H^1_{\Gamma_d}(\omega,\R^2)\times H^2_{\Gamma_d}(\omega)$ such that $\vect u^{h_n} \to (\vect w, v)$ 
  on a subsequence as $n\to\infty$ in the sense of Definition \ref{2.def:conv_displ} below.
  \item \emph{(Lower bound)} For every $(\vect u^{h_n})_n\subset H^1_{\Gamma_d}(\Omega,\R^3)$ sequence of equi-bounded energies
  such that $\vect u^{h_n} \to (\vect w, v)$, it holds
  \begin{equation*}\label{1.eq:lb}
  	\liminf_{n\to\infty}\energy_{h_n}(\vect{ u^{h_n}}) \geq \energy_0(\vect{w},v)\,.
  \end{equation*}
  \item \emph{(Upper bound)} For every $(\vect w, v)\in H^1_{\Gamma_d}(\omega,\R^2)\times H^2_{\Gamma_d}(\omega)$
   there exists  $(\vect u^{h_n})_n\subset H^1_{\Gamma_d}(\Omega,\R^3)$ such that
  \begin{equation*}\label{1.eq:ub}
   \vect u^{h_n} \to (\vect w, v) \quad\text{and}\quad 
   \lim_{n\to\infty}\energy_{h_n}(\vect{ u^{h_n}}) = \energy_0(\vect{w},v)\,.
  \end{equation*}
\end{enumerate}
\end{theorem}

\begin{remark}
In order to analyze a real problem, one can add forces in the above analysis 
(see Section \ref{subforces} below). Since we are dealing with the linearized elasticity, 
by the uniqueness argument, we can conclude converegence of the whole sequence in $(i)$. 
As shown in Section \ref{subforces}, Theorem \ref{thm:1} is proved even for perforated domains 
or domains with defects, when we don't have pointwise coercivity condition. 
\end{remark}

The second question addressed in this paper is about the local character of $\Gamma$-closure, i.e. characterization
of a composite of $N$ different constituents with prescribed volume fractions.   Such a problem has an
important application in the optimal design of materials \cite{Tor10}.
In the context of homogenization of elliptic equations and systems it is well known as  \emph{$G$-closure problem}. 
The first characterization of  $G$-closure by means of periodic homogenization for the case of the linear elliptic equation 
describing two isotropic materials has been performed independently in \cite{LuCh86} and \cite{Tar85}. 
They showed the \emph{locality property} --- every effective tensor obtained by mixing 
two materials with prescribed volume fractions can be locally (at a.e.~point) recovered as the pointwise limit of a sequence 
of periodically homogenized mixtures of the same volume fractions. Later on this was generalized to the case of nonlinear elliptic
and strongly monotone operators in divergence form \cite{Rai01}.
 A variational approach utilizing $\Gamma$-convergence method has been performed in \cite{BaBa09} and they proved the locality property for convex energies satisfying certain growth and coercivity assumption.
In the case of non-convex energies 
only a weaker result was obtained, leaving the characterization problem widely open. 
Our approach closely follows the one presented in \cite{BaBa09}, but it is not straightforward. Peculiarities arise 
through the usage of non-standard ``$\Gamma$-convergence'', which makes the diagonalization procedure more subtle, 
and through the fact that we also have to deal with dimension reduction. 
Although in the linear case the diagonalization argument can be performed by the metrizability 
property of $\Gamma$-convergence in the class of coercive functionals (see \cite{DMa93} for details), 
we retain in the setting below, because it also covers the cases of non-linear bending rod and 
von K\'arm\'an plate. 
Eventually we are able to give the local characterization of all possible effective behaviours of composite materials 
by means of in-plane periodically homogenized mixtures.
In this part we do not take into analysis perforated domains or  domains with defects, but assume uniform boundedness and coercivity assumption from \eqref{2.eq:A1}. 
Denote by $\set X^N(\Omega)$ the family of 
characteristic functions $(\chi_1,\ldots,\chi_N)\in L^\infty(\Omega,\{0,1\}^N)$ such that 
$\sum_{i=1}^N\chi_i(x) = 1$ for a.e.~$x\in\Omega$. Equivalently, there exists a measurable partition
$\{A_i\}_{i=1,\ldots,N}$ of $\Omega$ such that $\chi_i = \mathbbm{1}_{A_i}$ for $i=1,\ldots,N$.
Function $\chi\in\set X^N(\Omega)$ uniquely determines the composite material whose elastic energy is given by
\begin{equation}
\energy_\chi(\vect{u}) = \int_{\Omega}\sum_{i=1}^N Q_i(x,\sym\nabla_{h}\vect u)\chi_i(x)\dd x\,,
\end{equation}
where $Q_i$, $i=1,\ldots,N$, denote energy densities of constituents for which we assume the uniform boundedness and coercivity as in \eqref{2.eq:A1}. 

Next, we are considering a sequence of composites $(\chi^{h_n})_n$ parametrized by the material thickness $(h_n)_n$. 
Verifying the uniform boundedness and coercivity of the sequence of energy densities
\begin{equation*}
Q^{\chi^{h_n}} = \sum_{i=1}^N Q_i\chi_i^{h_n}\,,
\end{equation*}
we can perform the asymptotic analysis according to Theorem \ref{thm:1} (cf.~Section \ref{sec:3}). 
Given $(h_n)_n$ monotonically decreasing to zero and the sequence of mixtures $(\chi^{h_n})_n$, there exists a variational functional
$\Kchi$ such that on a subsequence (still denoted by $(h_n)_n$)
\begin{equation*}
\Kchi(\mat M,A) = \Kchim(\mat M,A) = \Kchip(\mat M,A)\,,
\quad \mat M\in\set S(\omega),\ A\subset\omega\text{ open with Lipschitz boundary},
\end{equation*}
where $\Kchim$ and $\Kchip$ are defined analogously as $\Khm$ and $\Khp$.
The limit energy density of the sequence of mixtures $(\chi^{h_n})_n$ then equals
\begin{equation} \label{bukal1b}
Q(x',\mat M_1,\mat M_2) =
\lim_{r\downarrow0}\frac{1}{|B(x',r)|}\Kchi(\mat M_1+x_3\mat M_2,B(x',r))\,,
\quad \forall\mat M_1, \mat M_2\in \R^{2\times2}_{\sym}\ \text{ and a.e.~}x'\in\omega\,.
\end{equation}
Before we proceed, let us recall results for periodic composites from \cite{NeVe13}. Assuming that the energy densities oscillate with
period $\eps(h)$ in in-plane directions, we obtain different limiting behaviour
depending on parameter $\gamma = \lim_{h\downarrow0}(h/\eps(h))$. For $\gamma\in(0,\infty)$,
an explicit formula for the homogenized energy density holds:
\begin{equation}\label{int:def.qhom}
Q_\gamma(\mat M_1,\mat M_2) = \inf_{\vect{\psi}\in H^1(\T \times I,\R^3)}
\sum_{i=1}^N\int_{A_i \times I}Q_i(x,\imath(\mat M_1 + x_3\mat M_2) + \nabla_\gamma\vect\psi)\dd x\,,
\end{equation}
where $\{A_i\}_{i=1,\ldots,N}$ is partition of $\T\times I$ and $\T$ denotes the two-dimensional torus, i.e.~$\T\simeq [0,1)^2$ 
with periodic boundary conditions. 
Let us only mention that the cases $\gamma=\infty$ and $\gamma=0$ can be obtained as poinwise limits 
of the energies when $\gamma \to \infty$ and $\gamma \to 0$, respectively. 

For $\theta\in[0,1]^N$ such that $\sum_{i=1}^N\theta_i=1$, let $\set Q_{\theta}$ denotes the
set of all quadratic forms $Q_{\gamma_k}$ from (\ref{int:def.qhom}) for some $\gamma_k\in(0,\infty)$
and partition $\{A_i^k\}_{i=1,\ldots,N}$ such that $|A_i^k| = \theta_i$ for all $i=1,\ldots,N$ and $k\in\N$.
Let $\set P_{\theta}$ denotes the closure of $\set Q_{\theta}$.
The following theorem provides the local characterization 
of the effective behaviour of a sequence of mixtures.
\begin{theorem}\label{thm:2}
Let $\Omega = \omega\times I$ with $\omega\subset\R^2$ Lipschitz domain.
Let  $(h_n)_n$ be monotonically decreasing to zero sequence
 and let $\theta\in L^\infty(\omega,[0,1]^N)$ 
 be such that $\sum_{i=1}^N \theta_i=1$ a.e.~in $\omega$. 
  Let 
$Q:\omega\times\R_{\sym}^{2\times2}\times\R_{\sym}^{2\times2}\to\R$ be given quadratic form (coercive and bounded).
The following statements are equivalent:
\begin{enumerate}[(i)]
  \item   there exists a subsequence, still denoted by $(h_n)_n$, such that there exists a sequence 
  of mixtures $(\chi^{h_n})_n\subset\set X^N(\Omega)$ whose limit energy density
  equals $Q$ and $\chi^{h_n} \overset{*}{\rightharpoonup}{\mu}$  for some $\mu \in L^\infty(\Omega,[0,1]^N)$ which satisfies $\int_I \mu_i (x',x_3)\, dx_3=\theta_i(x')$ a.e. in $\omega$, for $i=1,\dots,N$.
  \item $Q(x_0',\cdot,\cdot)\in \set P_{\theta(x_0')}$ for a.e.~$x_0'\in\omega$.
\end{enumerate}
\end{theorem}
\begin{remark}
	Since we do not take into account possible periodic out-of-plane oscillations, 
	only weaker claim is given in $(i)$, i.e.~$\int_I \mu \, \dd x_3 = \theta$. 
	In the case when we allow only in-plane oscillations without changing material along $x_3$ 
	direction,  implication $(ii) \implies (i)$ can be proved on the whole sequence.  
\end{remark}
\begin{remark}
The limit energy density obtained in 
\cite{Vel14a} for the case of von K\'arm\'an plate has the same form as \eqref{bukal1b}, 
where $Q_i=D^2W_i(I)$ and $W_i$ are stored energy densities of $i$-th material, for $i=1,\dots,N$. 
Also, the limit energy density for the bending rod obtained in \cite{MaVe14a} has a similar form.
Thus, Theorem \ref{thm:2} applies to these cases as well. 
\end{remark}
In Section \ref{sec:2} we introduce notation and some preparation definitions and results 
used in the subsequent Sections \ref{sec:3} and \ref{sec:4}, which are devised to the proofs 
of Theorems \ref{thm:1} and \ref{thm:2}, respectively. Auxilliary results are listed in the appendix.

\section{General framework}\label{sec:2}

\subsection{Notation}\label{sec:not}
For a vector $x=(x_1,x_2,x_3)\in\R^3$ by $x' = (x_1,x_2)$ we denote its first two components. Consequently, by $\nabla'$
we denote the gradient with respect to the first two variables $\nabla' = (\pa_1, \pa_2)$, 
while the standard gradient is denoted by $\nabla$. The scaled gradient is given by 
$\nabla_h = (\pa_1, \pa_2, \frac{1}{h}\pa_3)$ for some $h>0$. By $B(x,r)$ we denote the ball of radius $r$ around point $x \in \R^n$. 
Vectors, vector valued functions, as well as matrices and matrix valued functions representing displacement fields or their
gradients are denoted by bold symbols, while their components are indexed in subscripts. 
With the symbol $\wedge$ we denote the cross product of two vectors from $\R^3$.
For a matrix $\mat M$ we denote its symmetric part by $\sym\mat M = (\mat M + \mat M^{\tau})/2$. 
Operator $:$ denotes contraction of two matrices, i.e.~$\mat M:\mat K = \sum_{i,j}M_{ij}K_{ij}$.
Natural inclusion of $\R^{2\times 2}$ into $\R^{3\times 3}$ is given by
\begin{equation*}
\imath(\mat M) :=\sum_{i,j=1}^2{M}_{ij}(\vect e_i\otimes \vect e_j)\,,
\end{equation*}
where $(\vect e_1, \vect e_2,\vect e_3)$ denotes the canonical basis of $\R^3$. By $\mat I$ we denote the identity matrix. 
If $A \subset \R^n$, we denote by $\mathbbm{1}_A$ the characteristic function of the set $A$ 
and by $|A|$  its Lebesgue measure. 
If $A$ and $B$ are subsets of $\R^n$, by $A \ll B$ we mean
that  the closure $\bar{A}$ is contained in the interior $\inte(B)$ of $B$. 
\subsection{Variational functionals} 
In the sequel, let $(h_n)_{n\in\N}$ denotes a sequence which monotonically decreases to zero, 
$A\subset\omega$ be an open subset and 
$\mat M\in \set S(\omega)$. Recall the definition of the variational functionals 
from the Introduction, which resemble the definition of lower and upper $\Gamma$-limits:
\begin{definition}\label{defodK}
\begin{align}
\Khm(\mat M,A) &= \inf\Big\{\liminf_{n\to\infty}\int_{\Omega} \label{2.def:Khm}
		Q^{h_n}(x,\imath(\mat M\mathbbm{1}_{A \times I}) + \nabla_{h_n}\vect\psi^{h_n})\dd x\ | \\
 		& \qquad\ ~\quad(\vect\psi^{h_n})_n \subset H^1(\Omega,\R^3),\, (\psi_1^{h_n},\psi_2^{h_n},{h_n}\psi_3^{h_n}) \to 0 \text{ strongly in }L^2(\Omega,\R^3)\Big\}\,;\nonumber
\end{align}
\begin{align}
\Khp(\mat M,A) &= \inf\Big\{\limsup_{n\to\infty}\int_{\Omega} \label{2.def:Khp}
		Q^{h_n}(x,\imath(\mat M\mathbbm{1}_{A\times I}) + \nabla_{h_n}\vect\psi^{h_n})\dd x\ | \\
		& \qquad\ ~\quad (\vect\psi^{h_n})_n \subset H^1(\Omega,\R^3),\, (\psi_1^{h_n},\psi_2^{h_n},{h_n}\psi_3^{h_n}) \to 0 \text{ strongly in }L^2(\Omega,\R^3)\Big\}\,.\nonumber
\end{align}
\end{definition}
\noindent The above infimization is taken over all sequences of vector fields 
$(\vect{\psi}^{h_n})_{n}\subset H^1(\Omega,\R^3)$ such that \\
$(\psi_1^{h_n},\psi_2^{h_n},{h_n}\psi_3^{h_n}) \to 0$ strongly in the $L^2$-topology. 
From the definition it immediately follows
\begin{equation} \label{trivijala}
 \Khpm(\mat M\mathbbm{1}_{A \times I},\omega)=\Khpm (\mat M,A), \textrm{ for all } \mat M \in \set S (\omega),\, A \subset \omega \textrm{ open}. 
 \end{equation}
\begin{remark} \label{2.rem:top}
\begin{enumerate}[(a)] 
  \item Let $\set N(0)$ denotes the family of all neighbourhoods of $0$ in the strong $L^2$-topology, then 
  we have the standard  characterizations:
\begin{align}\label{bukal111111}
	\Khm(\mat M,A) 
	& = \sup_{\set{U}\subset\set{N}(0)}\liminf_{n\to\infty}\set K_{h_n}(\mat M, A,\set U)\,;\\ \nonumber
	\Khp(\mat M,A) & = \sup_{\set{U}\subset\set{N}(0)}\limsup_{n\to\infty}\set K_{h_n}(\mat M, A,\set U)\,,
\end{align}
where
\begin{equation}
\set K_{h_n}(\mat M, A,\set U) = \inf_{\substack{\vect{\psi}\in H^1(\Omega,\R^3)\\ 
		(\psi_1,\psi_2,h_n\psi_3)\in\set{U}}}\int_{\Omega}
Q^{h_n}(x,\imath(\mat M\mathbbm{1}_A) + \nabla_{h_n}\vect\psi)\dd x\,.
\end{equation}

\item Using the standard diagonalization argument, one can prove that infima in (\ref{2.def:Khm}) and (\ref{2.def:Khp})
are actually attained. 
\item 
 Additionally one can assume that the test sequences are equal to zero on $\partial \omega \times I$. It can be easily seen that this does not change the proof of Lemma \ref{lemma:equi-int1} (see Remark \ref{rempromjena}).
\item The following definitions are given in \cite{MaVe14a,Vel14a} 
\begin{align*}
	\Khtm(\mat M,A) &= \inf\Big\{\liminf_{n\to\infty}\int_{A \times I} 
	Q^{h_n}(x,\imath(\mat M) + \nabla_{h_n}\vect\psi^{h_n})\dd x\ | \, (\vect\psi^{h_n})_n \subset H^1(A \times I,\R^3)\\
	& \qquad\ ~\quad \, (\psi_1^{h_n},\psi_2^{h_n},{h_n}\psi_3^{h_n}) \to 0 \text{ strongly in }L^2(A \times I,\R^3)\Big\}\,;\nonumber
\end{align*}
\begin{align*}
	\Khtp(\mat M,A) &= \inf\Big\{\limsup_{n\to\infty}\int_{A \times I} 
	Q^{h_n}(x,\imath(\mat M) + \nabla_{h_n}\vect\psi^{h_n})\dd x\ | \, (\vect\psi^{h_n})_n \subset H^1(A \times I,\R^3)\\
	& \qquad\ ~\quad \, (\psi_1^{h_n},\psi_2^{h_n},{h_n}\psi_3^{h_n}) \to 0 \text{ strongly in }L^2(A \times I,\R^3)\Big\}\,.
\end{align*} 
Here we used Definition \ref{defodK} in order to include perforated domains and domains with 
defects (see Remark \ref{explanation1}). In the case when the coercivity condition \eqref{2.eq:A1} holds, 
these two definitions are equivalent, and relation \eqref{bukal111111} remains valid with 
(after replacing $\Khpm$ with $\Khtpm$)
\begin{equation*}
\widetilde{\set K}_{h_n}(\mat M, A,\set U) = \inf_{\substack{\vect{\psi}\in H^1(A \times I,\R^3)\\ 
		(\psi_1,\psi_2,h_n\psi_3)\in\set{U}}}\int_{A \times I}
Q^{h_n}(x,\imath(\mat M) + \nabla_{h_n}\vect\psi)\dd x\,.
\end{equation*}
One can additionally assume zero boundary condition for test functions in the case when 
$A \subset \omega$ has Lipschitz boundary. 	In \cite{Vel14a}  $	\Khtm(\mat M,A)$ and $\Khtp(\mat M,A)$ 
were defined for every $\mat M \in L^2(\Omega)$, but here we restrict the definition to the set of 
possible weak limits $\set S(\omega)$  (in the first variable). 
\end{enumerate}
\end{remark}
 Obviously, $\Khm(\mat M,A) \leq \Khp(\mat M,A)$. The equality can be proven in the similar way as in
 \cite[Lemma 3.7(a)]{Vel14a}, but on a subsequence of $(h_n)_n$, and we will give the sketch of the proof.
The basic novelty, given in \cite{MaVe14a,Vel14a}, compared with standard $\Gamma$-convergence techniques, consists in separating 
the fixed limit field $\mat M$ from the relaxation field given by the sequence $(\vect{\psi}^{h_n})$ 
(correctors) and exploring the 
functional which has this additional variable (the set of possible weak limits). 
This enables us to give an abstract definition of the energy and to prove that it possesses a 
quadratic energy density (see Lemma \ref{lemma:equi-int1} and Proposition \ref{prop:Int_form}). 
One of the key properties, which also exploits this separation property, 
is the continuity of functionals $\set{K}^\pm_{(h_n)}$ with respect to the first variable 
(see the proof of Lemma \ref{lem:ocjena} in the appendix, cf.~also \cite[Lemma 3.4]{Vel14a}):
\begin{eqnarray} \label{bukal1}
   	\left|\Khpm(\mat M_1,A) - \Khpm(\mat M_2,A)\right| \leq C(\alpha,\beta)\|\mat M_1 - \mat M_2\|_{L^2}
   	\left(\|\mat M_1\|_{L^2} + \|\mat M_2\|_{L^2}\right)\,,
   	\quad \forall \mat M_1\,,\mat M_2\in \set S(\omega)\,.
\end{eqnarray}
 In the following we state the key operating lemma.

\begin{lemma}\label{lemma:equi-int1}
Under the coercivity and uniform boundedness assumption \eqref{2.eq:A1}, for every sequence $(h_n)_n$, $h_n\downarrow0$, there exists a subsequence, still denoted by $(h_n)_n$, 
which satisfies
\begin{eqnarray*}\label{2.def:Kh_D}
& &\Kh(\mat M,A) := \Khm(\mat M,A) = \Khp(\mat M,A)\,, 
\quad \forall \mat M\in \set S(\omega)\,,\ 
\ \forall A\subset\omega\,\text{open}.
\end{eqnarray*}
For every 
$\vect{M}\in \set S(\omega) $ there exists a sequence of correctors 
$(\vect{\psi}^{h_n})_n\subset H^1(\Omega,\R^3)$, $(\psi^{h_n}_1,\psi^{h_n}_2,h_{n}\psi_3^{h_n})\to 0$ 
strongly in the $L^2$-norm and such that for every open $A\subset\omega$ we have
\begin{equation}\label{2:eq.minseq}
\Kh(\mat M,A) = \lim_{n\to\infty}\int_{A\times I} 
						Q^{h_n}(x,\imath(\mat M)  + \nabla_{h_n}\vect\psi^{h_n})\dd x\,,
\end{equation}
and the following properties hold:
\begin{enumerate}[(a)]
  \item 		the following decomposition is satisfied 
  		\begin{equation} \label{korektori1}
  			\vect{\psi}^{h_n}(x) = \left(\begin{array}{c} 0 \\ 0 \\ \dfrac{\varphi^{h_n} (x')}{h_n}\end{array}
  			\right) 
  			-x_3\left(\begin{array}{c} \nabla'\varphi^{h_n}(x') \\ 0 \end{array} \right)+\tilde{\vect{\psi}}^{h_n}\,,
  		\end{equation}
  	i.e.,
  		\begin{equation}\label{bukall3}
  			\sym \nabla_{h_{n}}\vect{\psi}^{h_n} = -x_3\imath(\nabla'^2\varphi^{h_n})
  				+ \sym\nabla_{h_{n}}\tilde{\vect{\psi}}^{h_n}\,,
  		\end{equation}
  		where 
  		\begin{align*}
  		{\varphi}^{h_n} \to 0\ &\text{ strongly in } H^1(\omega)\,,\quad
  		\tilde{\vect{\psi}}^{h_n} \to 0\ \text{ strongly in } L^2(\Omega,\R^3)\,,\quad and\\
  		&\limsup_{n\to\infty}\left(\|{\varphi}^{h_n}\|_{H^2} + \|\nabla_{h_{n}}\tilde{\vect{\psi}}^{h_n}\|_{L^2}
  		\right) \leq C(\Omega)(\|\mat M\|_{L^2}^2 + 1)\,, 
  		\end{align*}
 for some $C(\Omega)>0$;
  \item sequence $(|\sym\nabla_{h_{n}}\vect{\psi}^{h_n}|^2)_n$ is equi-integrable;
 \item for every $\mat M \in \set S(\omega) $ and $A \subset \omega$ open it holds  
  \begin{equation*} \label{minimabukal} 
 \Kh(\mat M, A)= \Kht(\mat M, A):=\Khtm (\mat M, A)=\Khtp (\mat M, A)\,;
  \end{equation*}
\item if for $\mat M \in \set S (\omega)$ a sequence 
$(\vect \zeta^{h_n})_n \subset H^1(\Omega, \R^3)$ satisfies \eqref{2:eq.minseq} for $A=\omega$ and 
$(\zeta^{h_n}_1,\zeta^{h_n}_2,h_{n}\zeta_3^{h_n})\to 0$ strongly in the $L^2$-norm, then 
\begin{equation*}\label{bukal1000}
\| \sym \nabla_{h_n} \vect \zeta^{h_n}-\sym \nabla_{h_n}\vect \psi^{h_n}\|_{L^2} \to 0\,. 
\end{equation*}
\end{enumerate}
\end{lemma}
\noindent The lemma can be proved by 
adapting the proof of \cite[Lemma 2.9]{MaVe14a}. Below we give
simple proof, which can be easily adapted to perforated domains and domains 
with defects. 

\begin{proof}
 Take a  
 countable set $\{ \mat M_j\}_{j \in N}\subset  \set S(\omega)$,
which is dense in $\set S(\omega)$ in the $L^2$-norm. 
Without loss of generality we assume that $\omega$ has smooth boundary. 
In the case when this is not satisfied we simply take $\tilde \omega \supset \omega$ 
with smooth boundary and extend each $Q^{h_n}$  on $\tilde \omega \backslash \omega$ by a 
constant quadratic form that satisfies coercivity and boundedness property.  
By a diagonal argument one can construct a subsequence, still denoted by $(h_n)_n$, 
such that on the subsequence we have 
\begin{equation}\label{bukal2}
\Kh(\mat M_j,\omega) := \Khm(\mat M_j,\omega) = \Khp(\mat M_j,\omega)\,, 
\quad \forall j \in \N\,,
\end{equation}
and such that the equality \eqref{2:eq.minseq} (for $A=\omega$) is valid for some sequence 
of correctors $(\vect \psi_{\mat M_j}^{h_n})_n \subset H^1(\Omega, \R^3)$ 
which satisfies $(\psi^{h_n}_{\mat M_j,1},\psi^{h_n}_{\mat M_j,2},h_{n}\psi_{\mat M_j,3}^{h_n})\to 0$ 
strongly in $L^2$.
Using the Griso's decomposition we can assume that they satisfy the property $(a)$ 
(see Lemma \ref{igor2}). This possibly requires a change of the sequence of correctors. 
By inequality \eqref{bukal1} it follows that 
 \begin{equation}
 \Kh(\mat M,\omega) := \Khm(\mat M,\omega) = \Khp(\mat M,\omega)\,, 
 \quad \forall \mat M \in \set S(\omega)\,,
 \end{equation}
and by the definition of $\Khp$ we can find a sequence of correctors 
$(\vect \psi_{\mat M}^{h_n})_n \subset H^1(\Omega,\R^3)$, 
which satisfies $(\psi^{h_n}_{\mat M,1},\psi^{h_n}_{\mat M,2},h_{n}\psi_{\mat M,3}^{h_n})\to 0$ 
strongly in $L^2$. Firstly, we prove equality \eqref{2:eq.minseq} for $A \subset \omega$ open 
with $|\partial A|=0$ or $|\partial A|>0$ and $\mat M=0 $ on $\partial A \times I$.
Suppose that the equality is not satisfied, then 
take an arbitrary subsequence of $(h_n)_n$, still denoted by $(h_n)_n$, such that there exists 
\begin{equation}
\label{postojilimes}
  \lim_{n\to\infty}\int_{A\times I} 
Q^{h_n}(x,\imath(\mat M)  + \nabla_{h_n}\vect\psi^{h_n}_{\mat M})\dd x\,. 
\end{equation}
 Notice that from the definition of $\Kh$, by testing with zero functions and using \eqref{2.eq:A1}, we can assume that for every $n \in \N$ 
\begin{equation}\label{omedjenost}
\|\sym \nabla_{h_n} \vect \psi_{\mat M}^{h_n}\|_{L^2} \leq C(\alpha,\beta) \|\mat M\|_{L^2}\,.
 \end{equation}
 Applying the Griso's decompostion and replacing the sequence of correctors, we obtain the decomposition in $(a)$. 
 Next, using Lemma \ref{ekvi1} and Lemma \ref{ekvi2} from Appendix, we replace 
 (on a subsequence) second gradients and scaled gradients by the corresponding 
 equi-integrable sequences, and form the sequence of correctors 
 $(\bar{\vect \psi}^{h_n}_{\mat M})_n \subset H^1(\Omega,\R^3)$ according to \eqref{korektori1}. 
 Notice that for the new sequence we have the following properties:
 \begin{equation}\label{svojstvanovih}
 (|\sym \nabla_{h_n} \bar{\vect \psi}^{h_n}_{\mat M}|^2)_n \textrm{ is equi-integrable and } 
 |\{\sym \nabla_{h_n} \bar{\vect\psi}^{h_n}_{\mat M} \neq \sym \nabla_{h_n} \vect\psi^{h_n}_{\mat M}\}|\to 0
  \textrm{ as } n \to \infty\,.
 \end{equation}
 As a consequence of \eqref{svojstvanovih}, on every subsequence, 
 still denoted by $(h_n)_n$, where the limits exist, the following inequalities hold:
 \begin{eqnarray}
 	\label{nej1}\lim_{n\to\infty}\int_{A\times I} 
 	Q^{h_n}(x,\imath(\mat M)  + \nabla_{h_n}\bar{\vect\psi}^{h_n}_{\mat M})\dd x & \leq & \lim_{n\to\infty}\int_{A\times I} 
 	Q^{h_n}(x,\imath(\mat M)  + \nabla_{h_n}\vect\psi^{h_n}_{\mat M})\dd x\,; \\
 	\label{nej2} \lim_{n\to\infty}\int_{(\omega \backslash \bar{A})\times I} 
 	Q^{h_n}(x,\imath(\mat M)  + \nabla_{h_n}\bar{\vect\psi}^{h_n}_{\mat M})\dd x & \leq & \lim_{n\to\infty}\int_{(\omega \backslash \bar{A})\times I} 
 	Q^{h_n}(x,\imath(\mat M)  + \nabla_{h_n}\vect\psi^{h_n}_{\mat M})\dd x\,; \\ \label{nej10}
 	\lim_{n\to\infty}\int_{\partial A\times I} 
 	Q^{h_n}(x,\imath(\mat M)  + \nabla_{h_n}\bar{\vect\psi}^{h_n}_{\mat M})\dd x &\leq& \lim_{n\to\infty}\int_{\partial A\times I} 
 	Q^{h_n}(x,\imath(\mat M)  + \nabla_{h_n}\vect\psi^{h_n}_{\mat M})\dd x\,.
 \end{eqnarray}
 From these relations, using the definition of $\Kh (\mat M,\omega)$, we conclude equalities in 
 \eqref{nej1}--\eqref{nej10}.
 Using Lemma \ref{nulizacija} from Appendix, we truncate the sequence 
 $(\bar{\vect \psi}^{h_n}_{\mat M})_n$ in set $A \times I$ such that the new sequence, 
 denoted by $(\bar{\vect \psi}^{h_n}_{0,\mat M})_n$, equals zero in a neighbourhood of 
 $\partial A \times I$ and zero value is assigned outside $A \times I$. 
 Next, we perform another truncation of the sequence $(\bar{\vect \psi}^{h_n}_{\mat M})_n$,
 but now in set $(\omega \backslash \bar{A}) \times I$, such that the new sequence, 
 denoted by $(\bar{\vect \psi}^{0,h_n}_{\mat M})_n$ equals zero in a neighbourhood of 
 $\partial (\omega \backslash \bar{A})\times I$ and zero value is assigned in $\bar{A}$. 
 By the definition of $\Kh(\mat M,A)$ and by testing with $(\bar{\vect \psi}^{h_n}_{0,\mat M})_n$  
 we conclude 
 \begin{equation} \label{nej3}
  \Kh(\mat M,A) \leq \lim_{n\to\infty}\int_{A\times I} 
 Q^{h_n}(x,\imath(\mat M)  + \nabla_{h_n}\bar{\vect\psi}^{h_n}_{0,\mat M})\dd x=\lim_{n\to\infty}\int_{A\times I} 
 Q^{h_n}(x,\imath(\mat M)  + \nabla_{h_n}\vect\psi^{h_n}_{\mat M})\dd x\,. 
 \end{equation}
Further on we take a sequence of correctors (on a subsequence) for 
$\mat M_A:=\mat M \mathbbm{1}_{A \times I}$, denoted by 
 $\vect\psi^{h_n}_{\mat M_A}$, find the equi-integrable substitution and perform the truncation 
 (in set $A \times I$ in a neighbourhood of $\partial A \times I$) 
 to obtain the new sequence of correctors $(\bar{\vect \psi}^{h_n}_{\mat M_A})_n$. Assigning 
 the zero value outside of $A\times I$ we easily conclude that (on a subsequence) 
 $$\Kh(\mat M,A)= \lim_{n\to\infty}\int_{A\times I} 
 Q^{h_n}(x,\imath(\mat M)  + \nabla_{h_n}\bar{\vect\psi}^{h_n}_{\mat M_A})\dd x\,.$$ 
 Define the new sequence of correctors by
 $$ 
 \dot{\vect{\psi}}^{h_n}_{\mat M}(x) = 
 \left\{\begin{array}{lr} \bar{\psi}^{h_n}_{\mat M_A}(x)\,, & \textrm{ if } x \in A \times I\,; \\   
 \bar{\psi}^{0,h_n}_{\mat M}(x)\,, & \textrm{ if } x \in (\omega \backslash \bar{A})\times I\,; \\ 
 0\, ,&  \textrm{ if } x \in \partial A \times I\,.   \end{array}  \right.     
 $$
 If the strict inequality in \eqref{nej3} would hold, then we would have 
 (because $\mat M \mathbbm{1}_{\partial A \times I}=0$ for a.e.~$x \in \Omega$)
 $$ \lim_{n\to\infty}\int_{\Omega} 
 Q^{h_n}(x,\imath(\mat M)  + \nabla_{h_n}\dot{\vect\psi}^{h_n}_{\mat M})\dd x<\Kh(\mat M,\omega)\,, $$
 which would yield a contradiction. The final claim is the consequence of the arbitrariness of the 
 subsequence, which we chose such that the limit \eqref{postojilimes} exists. 
 Notice that on the way, as a consequence of minimality, we also  proved that
 \begin{equation} \label{stefan11}
 \lim_{n \to \infty}\int_{(\omega \backslash A)\times I} 
 Q^{h_n}(x,\imath(\mat M)  + \nabla_{h_n}\vect\psi^{h_n}_{\mat M_A})\dd x=0\,,
 \end{equation}
 for every sequence of correctors $(\vect\psi^{h_n}_{\mat M_A})_{n\in \N}\subset H^1(\Omega,\R^3)$ 
 for $\mat M_A$, where $A \subset \omega$ open. Namely, we 
 can again take an arbitrary convergent subsequence, replace the correctors with equi-integrable 
 sequences, perform the truncation in a neighbourhood of $\partial A$ and assign values zero 
 outside of $A$. This would provide the argument in the same way as above. 
 The property $(b)$ is a consequence of the minimality property in the definition of $\Kh$, 
 see \cite[Lemma 2.9]{MaVe14a} (the proof uses the pointwise coercivity in \eqref{2.eq:A1}). 
 From the equi-integrability property it follows that equality \eqref{2:eq.minseq} is valid for every $A \subset \omega$ open and $\mat{M} \in \set S(\omega)$. 
To prove (c) we conclude as follows. Firstly, from \eqref{2:eq.minseq} and  \eqref{stefan11} it 
follows that $\Khtp(\mat M,A) \leq \Kh (\mat M, A)$. Next we prove that 
$\Khtm(\mat M, A)\geq \Kh (\mat M, A)$. 
If $A \subset \omega$ has $C^{1,1}$ boundary, then one can take a sequence of test functions 
for $\Khtm(\mat M, A)$ and replace it in the same way as above to conclude the inequality.
If $A \subset \omega$ is an arbitrary open set, then
from the definition of $\Khtpm$, we can easily conclude that 
\begin{equation*} 
\Khtm(\mat M\mathbbm{1}_{D\times I},A)=\Khtp(\mat M\mathbbm{1}_{D \times I},A)
=\Kh(\mat M,D)\,, \quad \forall \mat M \in \set S(\omega),\ \forall D \ll A\,,\ D \textrm{ with } C^{1,1} \textrm{ boundary},
\end{equation*}
(for more details see beginning of the proof of \cite[Lemma 3.7]{Vel14a}). From the latter, 
utilizing inequality \eqref{bukal1} for $\Kh$ and $\Khtpm$, which can be proved analogously, 
and \eqref{trivijala}  we conclude the claim $(c)$.
For the property $(d)$, see the proof of \cite[Lemma 2.9]{MaVe14a}.
\end{proof}

\begin{remark}\label{rempromjena}
	 One can, by the truncation argument if necessary, 
	 assume that for every $\mat M \in \set S(\omega)$, $n \in \N$, $\vect \psi^{h_n}=0$ 
	 in a neighbourhood of $\partial \omega \times I$, i.e.~$\varphi^{h_n}=0$ in a neighbourhood 
	 of $\partial \omega$ and $\tilde{\vect{\psi}}^{h_n}=0$ in a neighbourhood of 
	 $\partial \omega \times I$. To see that, one can first take such correctors for a countable 
	 set $\{\mat M_j\}\subset \set S(\omega) $, as at the beginning of the previous proof,  
	 then using the diagonalization argument and inequality \eqref{bukal1} 
	 construct a sequence of correctors for every $\mat M \in \set S(\omega)$, which consists of 
	 correctors for the sequence $\{\mat M_j\}$ (see \cite[Lemma 2.10]{MaVe14a}).
\end{remark}
We make the following assumption on which we will refer when necessary. 
\begin{assumption}\label{2:ass}
	For a sequence $(h_n)_n$ monotonically decreasing to zero we assume that it 
	satisfies the claim of Lemma \ref{lemma:equi-int1}.
\end{assumption}
\begin{remark}\label{bukal100000}
Assuming the Assumption \ref{2:ass} and using the representation \eqref{2:eq.minseq}, for every $\mat M \in \set S(\omega) $ we easily deduce the following:
	\begin{enumerate}[a)]
	  \item $\Kh(\mat M, \cdot )$ is subadditive  on open subsets of $\omega$, 
	  		i.e.~$\Kh(\mat M, A) \leq \Kh(\mat M,A_1)+\Kh(\mat M, A_2)$ for all open sets 
	  		$A, \ A_1,\ A_2 \subset \omega$ with $A \subset A_1 \cup A_2$;
	  \item $\Kh(\mat M, \cdot )$ is superadditive on open subsets of $\omega$,
	  		i.e.~$\Kh(\mat M,A) \geq \Kh(\mat M,A_1)+\Kh(\mat M,A_2)$ for all open sets 
	  		$A, \ A_1,\ A_2 \subset \omega$ with $A_1 \cup A_2 \subset A$ and $A_1 \cap A_2 =\emptyset$;
	  \item $\Kh(\mat M, \cdot)$ is monotone, i.e.~$\Kh (\mat M,A_1) \leq \Kh(\mat M, A_2)$, 
	  	for all open sets $A_1 \subset A_2 \subset \omega$, and inner regular, i.e. 		
		$$\Kh(\mat M,A)=\sup\{\Kh(\mat M,B): \, B \ll A\,,\ B\text{ open} \}\,, 		
		\quad \forall A \subset \omega \text{ open}\,.$$
	\end{enumerate}
	These properties imply that there exists a Borel measure $\mu_{\mat M}:\mathcal{B} \to [0,+\infty]$, 
	on the family $\mathcal{B}$ of Borel subsets of $\omega$,
 	such that $\Kh(\mat M,A)= \mu_{\mat M}(A)$ for every open set $A \subset \omega$ 
 	(see \cite[Theorem 14.23]{DMa93}). Moreover, the following properties of $\Kh(\cdot,\cdot)$ 
 	are easily deduced (see also the proof of \cite[Lemma 3.7]{Vel14a}):
		\begin{enumerate}[a)]
		  \setcounter{enumi}{3}
		  \item (homogeneity)
		  		$$ \Kh(t\mat M,A) =t^2 \Kh(\mat M,A),\quad \forall t \in \R\,,\ 
		  		\forall A \subset \omega \text{ open}\,,\ \forall \mat M \in \set S(\omega)\,; $$
		  \item (paralelogram inequality)
		  		\begin{align*}
		  			\Kh(\mat M_1+\mat M_2,A) + \Kh(\mat M_1-\mat M_2,A) &\leq 2\Kh(\mat M_1,A) 
		  				+ 2\Kh(\mat M_2,A)\,,\\
		  			&\qquad \qquad \forall A \subset \omega \text{ open}\,,\  
		  			\forall \mat M_1, \mat M_2 \in \set S(\omega)\,.
		  		\end{align*}
		\end{enumerate}
\end{remark}

\noindent Finally, we provide the result on integral representation of the variational functional, whose proof
follows the lines of the proof of \cite[Proposition 2.9]{Vel14a}, hence we omit it here. 
The key properties for proving that the energy density is a quadratic form are d) and f) from the above remark.

\begin{proposition}[Integral representation of $\Kh$]\label{prop:Int_form}
Let $(h_n)_n$, $h_n\downarrow0$, satisfies Assumption \ref{2:ass}. 
There exists a map $Q^0:\omega\times\R_{\sym}^{2\times2}\times\R_{\sym}^{2\times2} \to \R$ 
(depending on $(h_n)_n$)
such that for every open subset $A\subset\omega$ and every 
$\vect{M}\in \set S(\omega)$
\begin{equation}\label{2:prop:intrep}
\Kh(\mat M,A) = \int_A Q^0(x',\mat M_1(x'),\mat M_2(x'))\,\dd x'\,.
\end{equation}
Moreover, for a.e.~$x'\in\omega$, $Q^0(x',\cdot,\cdot)$ is uniformly bounded and coercive quadratic 
form satisfying 
\begin{eqnarray} \label{bbbbbbbb}
\textrm{(coercivity)}& & \frac{\alpha}{12} (|\sym \mat  M_1|^2+|\sym \mat M_2|^2) 
\leq Q^0 (x',\mat M_1,\mat M_2)\,; \\ \nonumber 
\textrm{(uniform boundedness)}& & Q^0 (x',\mat M_1,\mat M_2) 
\leq \beta (|\sym \mat M_1|^2+|\sym \mat M_2|^2)\,, \\ & & \nonumber\hspace{+10ex} 
\forall \mat M_1,\mat M_2 \in \R^{2 \times 2}\,, \textrm{ for a.e.~} x' \in \omega\,.
 \end{eqnarray} 
\end{proposition}

\begin{remark}
It is enough to prove equality (\ref{2:prop:intrep}) on countable dense families of open $C^{1,1}$ 
subsets of $\omega$ and symmetric $2\times2$ matrices. 
In fact this will be used in Section \ref{sec:4} in the characterization of the $\Gamma$-closure.
\end{remark}
\begin{remark} 
	Based on expression \eqref{2:eq.minseq}, one can easily conclude what would be the appropriate periodic or
	ergodic cell formula (cf. \cite[Section 4]{Vel14a}, see also \cite{buk16})), by using respectively  two-scale and stochastic two-scale
	convergence.
\end{remark}
\subsection{Perforated domains and domains with defects} 
An important application issue is relaxation of the pointwise a.e.~coercivity condition 
from \eqref{2.eq:A1} and motivation for that is to include, for example, perforated domains or 
domains with defects into our framework.
The situation we have in mind is when the coercivity condition is satisfied a.e.~only on a subset
$S^{h_n} \subset \Omega$ and part of the domain $\Omega \backslash S^{h_n}$ has defects. 
To be more precise we make the following technical assumption.
\begin{assumption}\label{ass: perforated}
	For every $n \in \N$, there exists a subdomain $S^{h_n} \subset \Omega$ and an extension 
	operator $E^{h_n}: H^1(S^{h_n},\R^3) \to H^1(\Omega,\R^3)$, which meet the following conditions 
	(for $\vect \psi \in H^1(S^{h_n},\R^3)$ , resp.~$\vect \psi \in H^1(\Omega,\R^3)$, 
	we denote by $\vect\psi^{E^{h_n}}=E^{h_n} \vect \psi$, 
	resp.~$\vect\psi^{E^{h_n}}=E^{h_n} \left(\vect \psi|_{S^{h_n}}\right)$):
	\begin{enumerate}[(a)]
		\item there exists  $\alpha>0$, such that for every $n \in \N$, 
		$\mat M \in \set S(\omega)$ and $\vect \psi \in H^1(\Omega, \R^3)$
		$$\alpha\|\mat M +\sym \nabla_{h_n}\vect\psi\|^2_{L^2(S^{h_n})} \leq  \int_{\Omega} 
		Q^{h_n}(x,\imath(\mat M) + \nabla_{h_n}\vect\psi)\dd x \,;   $$
		\item $$ \vect \psi^{E^{h_n}}|_{S^{h_n}}=\vect \psi|_{S^{h_n}}\,, \quad 
		\forall \vect \psi \in H^1(S^{h_n}, \R^3) \, ;$$
		\item
		$$ \| \sym \nabla_{h_n} \vect \psi^{E^{h_n}} \|_{L^2(\Omega)} 
		\leq C(\Omega) \| \sym \nabla_{h_n}\vect \psi\|_{L^2(S^{h_n})}\,,\quad 
		\forall \vect\psi \in H^1(S^{h_n},\R^3)\,; $$  
		\item for every sequence $(\vect\psi^{h_n})_n \subset H^1(\Omega,\R^3)$ such that 
		$\left( \|\sym\nabla_{h_n}\vect \psi^{h_n} \|_{L^2(S^{h_n})} \right)_n$  is bounded,
		\begin{equation*}
			 \|(\psi^{h_n}_1,\psi^{h_n}_2, h_n \psi^{h_n}_3)\|_{L^2} \to 0\quad  \implies  \quad
			 \left\|\left(\psi^{E^{h_n}}_1,\psi^{E^{h_n}}_2, h_n \psi^{E^{h_n}}_3\right)\right\|_{L^2} \to 0\,;         
		\end{equation*}
		\item for every $\mat M \in \set S(\omega)$ and $(\vect\psi^{h_n})_n \subset H^1(\Omega,\R^3)$,
		such that $\|(\psi^{h_n},\psi^{h_n}, h_n \psi^{h_n})\|_{L^2} \to 0 $ it holds
	$$\limsup_{n \to \infty} \left(\int_{\Omega} 
Q^{h_n}(x,\imath(\mat M) + \nabla_{h_n}\vect\psi^{h_n})\dd x  -\int_{\Omega} 
Q^{h_n}\left(x,\imath(\mat M) + \nabla_{h_n}\vect \psi^{E^{h_n}}\right)\dd x \right) \leq 0\,. $$
	\end{enumerate} 
\begin{remark} \label{remanalognoperf}
As a consequence of (a), (b) and (d), for every $\mat M_1, \mat M_2 \in \set S(\omega)$ and 
every $A \subset \omega$ open we have
\begin{eqnarray*} 
\left|\Khpm(\mat M_1,A) - \Khpm(\mat M_2,A)\right| \leq C(\Omega)\|\mat M_1 - \mat M_2\|_{L^2}
\left(\|\mat M_1\|_{L^2} + \|\mat M_2\|_{L^2}\right)\, .
\end{eqnarray*}
To see this, notice that in Definition \ref{defodK} of $\Khpm$, we can replace the sequence of 
test functions  $(\vect \psi^{h_n})_n \subset H^1(\Omega, \R^3)$ by the sequence 
$(\vect \psi^{E^{h_n}})_n$. 
Then, equality \eqref{bukal111111} is valid 
with the following definition of $\set K_{h_n}(\mat M, A,\set U)$:
\begin{equation*}
\set K_{h_n}(\mat M, A,\set U) = \inf_{\substack{\vect{\psi}\in H^1(\Omega,\R^3)\\ 
		\left(\psi^{E^{h_n}}_1,\psi^{E^{h_n}}_2, h_n \psi^{E^{h_n}}_3\right)\in\set{U}}}\int_{\Omega}
Q^{h_n}(x,\imath(\mat M\mathbbm{1}_A) + \nabla_{h_n}\vect\psi)\dd x\,.
\end{equation*}
To conclude the above inequality, one simply has to follow the proof of Lemma \ref{lem:ocjena} form Appendix.
\end{remark} 

\end{assumption}

\begin{example}\label{primjer1}
We assume that inside of the domain $\Omega_{h_n}$ there is a sequence of disjoint circular defects  
$\{B_{h_n}^i:=B(x_{h_n}^i,r_{h_n}^i)\}_{i=1}^{\infty}$ and  that there exists a constant $c>1$ such 
that for every $i \in \N$ we have  $K_{h_n}^i:=$  $ ^cB_{h_n}^i \backslash \overline{B_{h_n}^i} 
\subseteq \Omega_{h_n}$, where $  ^cB_{h_n}^i=B(x_{h_n}^i,cr_{h_n}^i)$. Furthermore, 
we assume that every $x \in \Omega_{h_n}$ is contained in at most $N$ balls of the sequence 
$\{ ^c B_{h_n}^i\}_{i=1}^{\infty}$. Denote by $R_{h_n}: \R^3 \to \R^3$ the 
mapping $R_{h_n}(x_1,x_2,x_3)=(x_1,x_2,\frac{x_3}{h_n})$ and by $R^{h_n}$ its inverse. 
For $i \in \N$, define the sets $K^{i,h_n}=R_{h_n} K_{h_n}^i$ and in the similar way define 
$B^{i,h_n},\, ^cB^{i,h_n} $. The regular set in $\Omega=R_{h_n} \Omega_{h_n}$ is defined by 
$S^{h_n} =\Omega \backslash \overline{ \cup_{i=1}^{\infty} B^i_{h_n}}$, where we assume 
the uniform boundedness and coercivity assumption from \eqref{2.eq:A1} for a.e.~$x \in S^{h_n}$. 
On $\Omega \backslash S^{h_n}$ we assume only the uniform 
boundedness condition with $\beta=\kappa_n$ for a.e.~$x \in \Omega \backslash S^{h_n}$ and 
$\kappa_n \to 0$. 

In the sequel, we explain construction of the extension operator. 
Denote by $K_1=B(0,c) \backslash \overline{B(0,1)} $.
Using \cite[Lemma 4.1]{OlShYo92}, there exists a linear operator 
$P:H^1\left(K_1,\R^3\right)$ $\to H^1\left(B(0,c),\R^3\right)$ and $C>0$  such that
\begin{eqnarray} \label{prva}
 \|P\vect w\|_{H^1} &\leq& C\left(\|\vect w \|_{L^2(K_1)} +\|\sym \nabla \vect w\|_{L^2(K_1)}  \right)\,;\\  
 \label{druga}  \|\nabla (P\vect w)\|_{L^2}  &\leq& C \| \nabla (P\vect w)\|_{L^2(K_1)}\,;\\ 
 \label{treca}  \|\sym \nabla (P\vect w)\|_{L^2}  &\leq& C \|\sym \nabla (P\vect w)\|_{L^2(K_1)}\,.  
\end{eqnarray}

Using the operator $P$, define by rescaling and translation operators 
$P^i_{h_n}: H^1(K^i_{h_n},\R^3) \to H^1(B^i_{h_n},\R^3)$, $i\in\N$. 
Define operators $P^{i,h_n}:H^1(K^{i,h_n},\R^3) \to H^1(B^{i,h_n},\R^3) $ 
by rescaling naturaly the operators $P_{h_n}^i$: for $\vect \psi \in H^1(K^{i,h_n},\R^3)$, set 
$P^{i,h_n} \vect \psi= P_{h_n}^i(\vect \psi \circ R_{h_n})\circ R^{h_n} $. 
Finally, define the extension operator $E^{h_n} :H^1(S^{h_n},\R^3) \to H^1(\Omega,\R^3)$, 
such that for every $i \in \N$,
$$  E^{h_n} \vect \psi|_{B^{i,h_n}} 
= P^{i,h_n}\vect\psi\,, \quad \forall \vect \psi \in H^1 (S^{h_n},\R^3)\,.$$
 We want to prove that the operator $E^{h_n}$ meets conditions (b)-(e) from Assumption 
 \ref{ass: perforated}. Condition (b) is direct, and (c) is a consequence of \eqref{treca} and 
 the fact that every $x \in \Omega_{h_n}$ is contained in at most $N$ balls of the sequence 
 $\{ ^c B_{h_n}^i\}_{i=1}^{\infty}$. Condition (e) is a consequence of (c) and the fact that 
 $\kappa_n \to 0$. To prove (d), we proceed as follows. 
 Integrating the Korn's inequality with the boundary condition \eqref{Kornwithbc} from $1$ to $c$, 
 we obtain that there exists a constant $C>0$, such that
 \begin{equation*}
 \|\vect w\|^2_{L^2 (B(0,1))} 
 \leq C \left(\|\vect w\|^2_{L^2(K_1))} + \|\sym \nabla \vect w\|^2_{L^2 (B(0,c))}  \right)\,,
  \quad \forall \vect w \in H^1(B(0,c),\R^3)\,. 
 \end{equation*} 
After rescaling on $ ^cB^{i}_{h_n}$, for every $i \in \N$ and $h_n$ we obtain 
\begin{equation*}
\|\vect \psi  \|^2_{L^2( B^{i}_{h_n})} 
\leq C \left( \|\vect \psi  \|^2_{L^2( K^{i}_{h_n})}+(r^i_{h_n})^2 
 \|\sym \nabla \vect \psi\|^2_{L^2( ^ c B^{i}_{h_n})}\right)\,, \quad 
 \forall  \vect \psi \in H^1(^ c B^{i}_{h_n},\R^3)\,. 
\end{equation*}
Rescaling once again by $R_{h_n}$ on $\Omega$, for every $i \in \N$, and $h_n$ we obtain 
\begin{equation} \label{zadnjeprimjer}
\|\vect \psi  \|^2_{L^2( B^{i,h_n})} \leq 
C \left( \|\vect \psi  \|^2_{L^2( K^{i,h_n})}+(r^i_{h_n})^2 
 \|\sym \nabla_{h_n} \vect \psi\|^2_{L^2( ^ c B^{i,h_n})}\right)\,, 
 \quad \forall  \vect \psi \in H^1(^ c B^{i,h_n},\R^3)\,.
\end{equation}
Notice that $\forall i \in \N$ it holds $r^i_{h_n} <h_n$. From \eqref{zadnjeprimjer}, property (c) 
and the fact that $(\psi_1^{h_n},\psi_2^{h_n}) \to 0$ strongly in the $L^2$-norm, 
we easily conclude that $(\psi^{E^{h_n}}_1, \psi^{E^{h_n}}_2) \to 0$ strongly in the $L^2$-norm. 
It remains to show that $h_n\psi^{E^{h_n}}_3 \to 0$  strongly in the $L^2$-norm. 
Define $\hat{\psi}_3^{E^{h_n}}$ and $\bar{\psi}_3^{E^{h_n}}$ by
$$  \hat{\psi}_3^{E^{h_n}}=\int_I \psi_3^{E^{h_n}} \dd x_3\,, 
\quad  \bar{\psi}_3^{E^{h_n}}= \psi_3^{E^{h_n}}-\hat{\psi}_3^{E^{h_n}}\,.$$
First, we use the property (b) and the Griso's decomposition of $\vect \psi^{E^{h_n}}$ 
(see \eqref{griso1}) to conclude from \eqref{prvaKorn} (see also \eqref{3.eq:compact_est}) 
that $\partial_{\alpha} (h_n \hat{\psi}_3^{E^{h_n}}) \to 0$ strongly in the $L^2$-norm 
for $\alpha=1,2$. 
Since from \eqref{prvaKorn} we easily obtain that $\partial_{i} (h_n \bar{\psi}_3^{E^{h_n}}) \to 0$ 
strongly in the $L^2$-norm for $i=1,2,3$, we have that $\partial_{i} (h_n \psi_3^{E^{h_n}}) \to 0$ 
strongly in the $L^2$-norm. The final claim follows from the fact that there exists a constant $C>0$, 
such that $\forall i \in \N$ and $\forall h_n$
 \begin{equation}\label{bukalbravo} 
\|\eta \|^2_{L^2( B^{i}_{h_n})} \leq C \left( \|\eta  \|^2_{L^2( K^{i}_{h_n})}+(r^i_{h_n})^2  
\| \nabla \eta\|^2_{L^2( ^ c B^{i}_{h_n})} \right)\,, \quad \forall   \eta \in H^1(^ c B^{i}_{h_n})\,,
\end{equation}                        
which can be proved analogously as \eqref{zadnjeprimjer}, and the fact that $h_n \psi^{h_n}_3 \to 0$ 
strongly in the $L^2$-norm. Observe that, although we assumed circular defects above, 
the only important thing is to control the geometry of parts with defects, i.e.~to have a 
control over constants in the expressions 
\eqref{treca},\eqref{zadnjeprimjer}, \eqref{bukalbravo}. Thus, other geometries of defects
are also possible (see Figure \ref{fig:perfdomain}).
\end{example}
\begin{figure}[!h] 
\centering
	\subfloat[]{
		\includegraphics[height = 20mm]{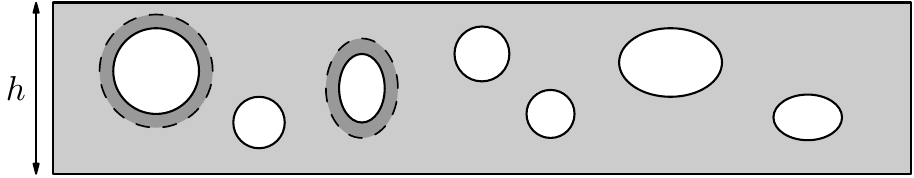}   
		\label{fig:perfdomain}} \hspace{7mm} 
	\subfloat[]{
		\includegraphics[height = 20mm]{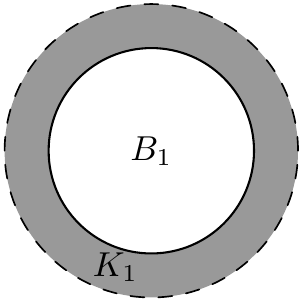} 
		\label{fig:krug}} \hspace{7mm} 
	\subfloat[]{
		\includegraphics[height = 20mm]{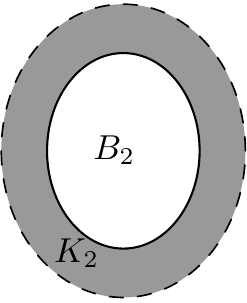}
		\label{fig:elipsa}}
	\caption{(a) Thin domain with defects.    
		(b) Extension operator of type I. (c) Extension operator of type II.}
	\label{fig:perfdomain}
\end{figure}

\begin{example}
With an analogous analysis as in the previous example, one can analyze thin domain with oscillatory 
boundary (see Figure \ref{fig:zubdomain}). Here we assume that the size of oscillations
$\varepsilon \leq h$. 
The extension operator can be constructed in a similar way as before, 
using the extension operator $P : H^1(Q_1,\R^3) \to H^1(Q,\R^3)$, where $Q=Q_1 \cup Q_2$.
\end{example}
\begin{figure}[!h]
\centering
	\subfloat[]{
		\includegraphics[height = 25mm]{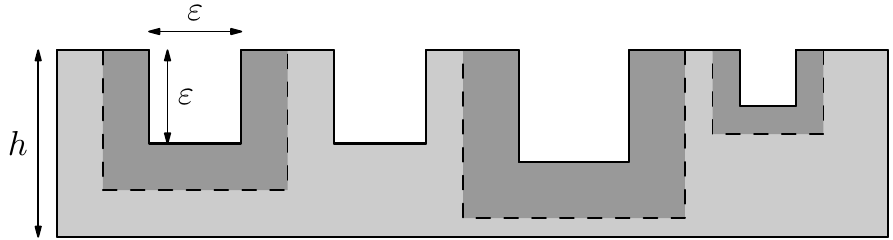} 
		\label{fig:zubdomain}} \hspace{15mm}
	\subfloat[]{
		\includegraphics[height = 20mm]{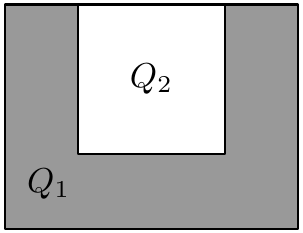} 
		\label{fig:udomain}}
	\caption{(a) Thin domain with oscillatory boundary.    
		(b) Extension operator.}  
	\label{fig:composite}
\end{figure}

We now state a lemma for domains with defects, which is analogous to 
Lemma \ref{lemma:equi-int1}. 
\begin{lemma}\label{lem:perforated}
Under the uniform boundedness assumption from \eqref{2.eq:A1} and the Assumption \ref{ass: perforated}, for every sequence $(h_n)_n$, $h_n\downarrow0$, there exists a subsequence, still denoted by $(h_n)_n$, 
which satisfies
\begin{equation*}
	\Kh(\mat M,A) := \Khm(\mat M,A) = \Khp(\mat M,A)\,, \quad \forall \mat M\in \set S(\omega)\,,
	\ \forall A\subset\omega\,\text{open}\,.
\end{equation*}
For every 
$\vect{M}\in \set S(\omega)$, there exists a sequence of correctors 
$(\vect{\psi}^{h_n})_n\subset H^1(\Omega,\R^3)$, $(\psi^{h_n}_1,\psi^{h_n}_2,h_{n}\psi_3^{h_n})\to 0$ 
strongly in the $L^2$-norm, and  for every open subset $A\subset\omega$, 
which satisfies $\mat M\mathbbm{1}_{\partial A \times I}=0$ for a.e.~$x \in \Omega$, we have
\begin{equation}\label{2:eq.minseq1}
\Kh(\mat M,A) = \lim_{n\to\infty}\int_{A\times I} 
Q^{h_n}(x,\imath(\mat M)  + \nabla_{h_n}\vect\psi^{h_n})\dd x\,.
\end{equation}
In the case when $\mat M\mathbbm{1}_{\partial A \times I} \neq 0$, we can only conclude that the 
left-hand side is less or equal to the right-hand side in \eqref{2:eq.minseq1}.
The following properties hold:
\begin{enumerate}[(a)]
	\item 		the following decomposition holds: 
	\begin{equation*} 
	\vect{\psi}^{h_n}(x) = \left(\begin{array}{c} 0 \\ 0 \\ \dfrac{\varphi^{h_n} (x')}{h_n}\end{array}
	\right) 
	-x_3\left(\begin{array}{c} \nabla'\varphi^{h_n}(x') \\ 0 \end{array} \right) + 
	\tilde{\vect{\psi}}^{h_n}\,,
	\end{equation*}
	i.e.
	\begin{equation*}\label{bukall3}
	\sym \nabla_{h_{n}}\vect{\psi}^{h_n} = -x_3\imath(\nabla'^2\varphi^{h_n})
	+ \sym\nabla_{h_{n}}\tilde{\vect{\psi}}^{h_n}\,,
	\end{equation*}
	where 
	\begin{align*}
		{\varphi}^{h_n} \to 0\ &\text{ strongly in } H^1(\omega)\,,\quad
		\tilde{\vect{\psi}}^{h_n} \to 0\ \text{ strongly in } L^2(\Omega,\R^3)\,,\quad and\\
		&\limsup_{n\to\infty}\left(\|{\varphi}^{h_n}\|_{H^2} + \|\nabla_{h_{n}}\tilde{\vect{\psi}}^{h_n}\|_{L^2}
		\right) \leq C(\Omega)(\|\mat M\|_{L^2}^2 + 1)\,, 
	\end{align*}
	for some $C(\Omega)>0$;
	\item for every subseqeunce, still denoted by $(h_n)_n$, one can change the sequence of correctors, if necessary, on a further subsequence 
	(possibly depending on $\mat M$), such that 
	the sequences $(|\nabla' \varphi^{h_n}|^2)_n $ and $(|\nabla_{h_{n}}\tilde{\vect{\psi}}^{h_n}|^2)_n$
	are equi-integrable, i.e.~the sequence $(|\sym \nabla_{h_{n}}\vect{\psi}^{h_n}|^2)_n$ is 
	equi-integrable. On that subsequence for the changed correctors, formula \eqref{2:eq.minseq1} is 
	valid for every open subset $A \subset \omega$. 
	\item for every $\mat M \in \set S(\omega) $ and $A \subset \omega$ open, we have  
	\begin{equation*} \label{minimabukal} 
		\Khtm (\mat M, A)\leq \Khtp (\mat M, A)\leq \Kh(\mat M, A)\,.
	\end{equation*}
\end{enumerate}
\end{lemma}
\begin{proof}
The proof follows the lines of the proof of Lemma \ref{lemma:equi-int1}, after replacing the 
sequence of test functions $(\vect \psi_{\mat M}^{h_n})_n \subset H^{1}(\Omega, \R^n)$ by 
$(\vect \psi_{\mat M}^{E^{h_n}})_n \subset H^{1}(\Omega, \R^n)$ and using Remark \ref{remanalognoperf}. 
The validity of statements (b), (c) and (d) from Lemma \ref{lemma:equi-int1} cannot be concluded, 
since there the pointwise a.e.~coercivity was used. However, weaker statements for the 
equi-integrability and relation between $\Khtm,\Khtp,\Kh$ can be concluded without the 
pointwise coercivity, in the same way as in the proof of Lemma \ref{lemma:equi-int1}. 
\end{proof}	
\begin{remark}\label{explanation1}
The main obstruction in proving equality in (c) is that from the definition of $\Khtpm$, one cannot 
conclude that the minimizing sequence has bounded symmetric scaled gradients and thus, one 
cannot perform the analysis analogous to the one given in \cite{Vel14a}. 
On the other hand, the definition of $\Khpm$ takes the whole domain $\Omega$ into account, 
where an extension operator exists, and thus, is more suitable for domains with defects. 
\end{remark}
\begin{remark}\label{remslikamario}
Validaty of Remarks \ref{rempromjena} and \ref{bukal100000} can be easily checked in the same way as before, 
and the claim of Proposition \ref{prop:Int_form} is satisfied, but without the coercivity property
in \eqref{bbbbbbbb}. 
\end{remark}
\section{Convergence of functionals and proof of Theorem \ref{thm:1}}\label{sec:3}
\subsection{Convergence of energy functionals}
Here we precisely explain what we mean by $\energy_{h_n}(\vect{u}^{h_n}) \to \energy_0(\vect w, v)$ as $n\to\infty$.
First, let us introduce the space of limiting displacements
\begin{equation*}
\set U(\omega) = \{(\vect w, v)\, :\, \vect w \in H^1_{\Gamma_d}(\omega,\R^2)\,, v\in H^2_{\Gamma_d}(\omega)\}\,.
\end{equation*}
\begin{definition}[Convergence of displacements]\label{2.def:conv_displ}
We say that a sequence of displacement fields $(\vect{u}^{h_n})_n\subset H^1_{\Gamma_d}(\Omega,\R^3)$ converges to 
a pair of in-plane and out of-plane displacements $(\vect w, v)\in\set U(\omega)$, 
in notation $\vect{u}^{h_n} \to (\vect w, v)$, 
if there exists decomposition of $\vect{u}^{h_n}$ of the form
\begin{equation}\label{2.eq:displ_form}
\vect{u}^{h_n}(x) = \left(\begin{array}{c}\vect{w}(x') \\ \dfrac{v(x')}{h_n}\end{array}\right) 
						- x_3\left(\begin{array}{c} \nabla'v(x') \\ 0 \end{array} \right)
						+ {\vect{\psi}}^{h_n}(x)\,,
\end{equation}
where the {\em corrector sequence} $(\vect\psi^{h_n})_n\subset H^1_{\Gamma_d}(\Omega,\R^3)$ satisfies 
$\lim_{n\to\infty}(\psi_1^{h_n}, \psi_2^{h_n},h_n\psi_3^{h_n}) = 0$ in the $L^2$-norm. 
\end{definition}

\begin{remark}\label{bukal100}
Otherwise stated, the above formulation means that $(u_1^{h_n},u_2^{h_n})^T \to \vect w - x_3\nabla'v$  
and $h_nu_3^{h_n}\to v$ in the $L^2$-norm as $n\to\infty$. This non-intuitive definition of convergence will be 
clarified in Section \ref{sec:3} below. 
It is easy to check that the convergence is properly defined, i.e.~the limit (if exists) is unique. 
Usually, in the linearized elasticity one does the scaling of the unknowns (see, e.g.~\cite{Cia97}) 
and then easily concludes that the
limit model is the Kirchoff-Love model (the first two terms in the expression \eqref{2.eq:displ_form}).
However, here we prefer not to scale the unknowns apriori. 
One of the basic reason is that we would then unnecessary introduce $h$ in the left hand side 
of \eqref{bukall3}. Moreover, Griso's decomposition gives us all the information we need: the limit 
field, which satisfies the Kirchoff-Love ansatz, and the corrector field. 
\end{remark}

\begin{definition}[Convergence of energies]
We say that a sequence of elastic energies $(\energy_{h_n})_n$  converges to the limit
energy $\energy_0$ if for all $(\vect w,v)\in\set U(\omega)$, the following two statements hold:
\begin{enumerate}[(i)]
  \item ($\liminf$ inequality) for every sequence $(\vect{u}^{h_n})_n \subset H^1_{\Gamma_d}(\Omega,\R^3)$ converging to $(\vect w, v)$,
  in the sense of Definition \ref{2.def:conv_displ},
  \begin{equation*}
  	\liminf_{n\to\infty}\energy_{h_n}(\vect u^{h_n}) \geq \energy_0(\vect w,v)\,;
  \end{equation*}
  \item ($\limsup$ inequality) there exists a sequence $(\vect{u}^{h_n})_n \subset H^1_{\Gamma_d}(\Omega,\R^3)$ converging to $(\vect w, v)$
  such that 
  \begin{equation*}
  	\limsup_{n\to\infty}\energy_{h_n}(\vect u^{h_n}) \leq \energy_0(\vect w,v)\,.
  \end{equation*}
\end{enumerate}
\end{definition}
\begin{remark}
Like in the standard $\Gamma$-convergence, condition (ii) from the previous definition can be rephrased in terms of
the existence of a {\em recovery sequence}, i.e.~there exists a sequence $(\vect{u}^{h_n})_n \subset H^1_{\Gamma_d}(\Omega,\R^3)$ converging to $(\vect w, v)$
such that 
  \begin{equation*}
  	\energy_0(\vect w,v) = \lim_{n\to\infty}\energy_{h_n}(\vect u^{h_n})\,.
  \end{equation*}
\end{remark}

\subsection{Compactness}\label{sec3:comp} 
Let $(h_n)_n$ be a sequence of plate thickness and 
let $(\vect{u}^{h_n})_n\subset H^1_{\Gamma_d}(\Omega,\R^3)$ be a sequence of displacement fields, which
equal zero on $\Gamma_d\times I\subset\pa\omega\times I$ of strictly positive surface measure.
Assume that $(\vect{u}^{h_n})_n$ is the sequence of equi-bounded energies, 
i.e.~$\limsup_{n\to\infty}\energy_{h_n}(\vect{u}^{h_n}) \leq C < \infty$.
Then the coercivity condition \eqref{2.eq:A1}  implies the equi-boundednes
of the sequence of the symmetrized scaled gradients
\begin{equation*}
\limsup_{n\to\infty}\|\sym\nabla_{h_n}\vect{u}^{h_n}\|_{L^2} < \infty\,.
\end{equation*}
Applying Lemma \ref{app:lem.limsup} from Appendix we obtain the compactness result.

\subsection{Lower bound}\label{sec3:lb}
Let $(\vect{u}^{h_n})_n\subset H^1_{\Gamma_d}(\omega\times I,\R^3)$ be the sequence of displacements of 
equi-bounded energies and $(h_n)_n$ satisfies Assumption \ref{2:ass} or Lemma \ref{lemma:equi-int1}. 
 We  decompose $\vect{u}^{h_n}$ (on a subsequence)
\begin{align}
\vect{u}^{h_n}(x) = \left(\begin{array}{c}\vect{w}(x') \\ \dfrac{v(x')}{h_n}\end{array}\right) 
						- x_3\left(\begin{array}{c} \nabla'v(x') \\ 0 \end{array} \right)\nonumber
				+ \vect{\psi}^{h_n}(x). 
\end{align}
where $\vect{w}\in H_{\Gamma_d}^1(\omega,\R^2)$ and 
$v\in H_{\Gamma_d}^2(\omega)$ are fixed horizontal and vertical displacemets and $(\psi^{h_n}_1,\psi^{h_n}_2, h_n \psi^{h_n}_3) \to 0$ in the $L^2$-norm. 
The proof of lower bound immediately follows from the definition of the functional $\mathcal{K}_{(h_n)}$. 
\subsection{Construction of the recovery sequence}\label{sec3:ub}
Let $(\vect{w}, v)\in \set{U}(\omega)$ and $(h_n)_n$ satisfies Assumption \ref{2:ass} or Lemma \ref{lem:perforated}. Utilizing Lemma \ref{lemma:equi-int1} (and Remark \ref{rempromjena}, i.e., Rematk \ref{remslikamario}), 
there exist a sequence  
$\vect{\psi}^{h_n}\subset H^1(\Omega,\R^3)$, such that $\vect{\psi}^{h_n}=0$ on $\partial \omega \times I$,  $(\psi_1^{h_n},\psi_2^{h_n},h_n\psi_3^{h_n})\to 0$ in the $L^2$-norm and 
\begin{equation*}
\int_\omega Q^0(x',\sym\nabla\vect{w},-\nabla'^2v)\dd x' = 
\lim_{n\to\infty}\int_{\omega\times I}Q^{h_{n}}(x, \imath(\sym\nabla\vect{w} - x_3\nabla'^2v) 
+ \sym\nabla_{h_{n}}{\vect{\psi}}^{h_{n}})\dd x\,.
\end{equation*}
Defining the sequence of displacements $\vect u^{h_n}$ by
\begin{equation*}
\vect{u}^{h_n} = \left(\begin{array}{c} \vect w \\ \dfrac{v}{h_n}\end{array}\right)
 - x_3 \left(\begin{array}{c} \nabla'v \\ 0 \end{array}\right) + \vect{\psi}^{h_n}\,,
\end{equation*}
we have that $\vect u^{h_n}\to (\vect w, v)$ as $n\to \infty$ in the sense of 
Definition \ref{2.def:conv_displ}, and the previous identity implies 
\begin{equation*}
\energy_0 (\vect{w},v) = \lim_{n\to\infty}\energy_{h_n}(\vect u^{h_n})\,,
\end{equation*}
which finishes the proof.
\begin{remark} 
	Analysis of the homogeneous isotropic plate (see \cite{Cia97}), after rescaling, reduces to the $\Gamma$-convergence
	problem in the $L^2$-topology.
	The limit energy density (written on $3D$ domain) depends on $\sym \nabla (\vect w-x_3 \nabla v)$, 
	where $\vect w-x_3 \nabla v$ is (as above) the limit of first two components of the $h_n$ problem $(u^{h_n}_1,u^{h_n}_2)$. 
	Thus, the limit energy, although it contains the second gradients, can be interpreted as of the same 
	type as the starting one, but restricted to the appropriate space. This  does not happen here, 
	since in the limit energy we cannot separate membrane energy (dependent on $\sym \nabla \vect w$)
	from the curvature energy (dependent on $\nabla^2 v$). 
\end{remark} 
\subsection{Adding lower order term with forces}\label{subforces}
It is relatively easy to include into the model and analyze external forces of the form $(f_1,f_2,h_nf_3)$, 
which corresponds to the standard scaling for plates in the context of linearized elasticity, see 
\cite{Cia97}. 
Define the energy functional 
\begin{equation} \label{dodanosile}
\fenergy_{h_n}(\vect u^{h_n})= \int_{\Omega}Q^{h_n}(x,\sym \nabla_{h_n} \vect u^{h_n}) \dd x 
- \int_{\Omega}\vect f \cdot (u_1^{h_n}, u_2^{h_n},h_n u_3^{h_n})\dd x\,.
\end{equation}
For the sequence of minimizers, denoted by $(\vect u_m^{h_n})_n $, one easily deduces, using the  
pointwise coercivity from \eqref{2.eq:A1}, Corollary \ref{kornthincor} and testing the functional 
$\fenergy_{h_n}$ with zero function,
\begin{equation}
\|\sym \nabla_{h_n} \vect u_m^{h_n} \|^2_{L^2}\leq C\| (u_1^{h_n}, u_2^{h_n},h_n u_3^{h_n})\|_{L^2} 
\leq C\|\sym \nabla_{h_n} \vect u_m^{h_n} \|_{L^2}\,.
\end{equation}
After obtaining the boundednes of the sequence 
$\left(\|\sym \nabla_{h_n} \vect u_m^{h_n} \|_{L^2}\right)_n$, the analysis goes the same way 
as in the previous section. The above analysis implies in the standard way that the minimizers of 
$h_n$-problem converge in the sense of Definition \ref{2.def:conv_displ} (on a whole sequence) 
to the minimizer of the limit problem, when the pointwise coercivity assumption from \eqref{2.eq:A1} 
is satisfied. The following assumption is needed for the convergence of minimizers in the case when 
the pointwise coercivity assumption from \eqref{2.eq:A1} is not satisfied, but Assumption 
\ref{ass: perforated} is assumed. 
\begin{assumption}\label{ass;prfff}
	We assume the validity of Assumption \ref{2:ass} and that for $S^{h_n} \subset \Omega$ additionally the extension operator 
	$E^{h_n}: H^1(S^{h_n},\R^3) \to H^1(\Omega,\R^3)$ satisfies the following conditions: 
	\begin{enumerate} [(a)]
		\item 	$\vect\psi=0  \textrm{ on } \Gamma_d \times I \, \implies  
				\vect\psi^{E^{h_n}}=0 \textrm{ on }  \Gamma_d \times I;$
		\item $1_{S^{h_n}} \rightharpoonup \theta \in L^{\infty}(\Omega,[0,1]); $
		\item $(f_1,f_2, h_n f_3)1_{\Omega \backslash S^{h_n}}= 0$  or for every sequence 
			$(\vect u^{h_n})_n \subset H^1_{\Gamma_d}(\Omega,\R^3)$, there exists a constant $C>0$, such that
			for every $n\in\N$,
		\begin{align*}
			\int_{\Omega} Q^{h_n}(x, \sym \nabla_{h_n} \vect u^{h_n} ) \dd x & \leq C \left(\|(u^{h_n}_1,u^{h_n}_2, h_n u^{h_n}_3)\|_{L^2}+1\right) \\
			 & \quad	\implies   (u^{h_n}_1,u^{h_n}_2, h_n u^{h_n}_3)-\left(u^{E^{h_n}}_1,u^{E^{h_n}}_2, h_n u^{E^{h_n}}_3\right) \to 0 \textrm{ strongly in } L^2\,.               
		\end{align*}
	\end{enumerate} 
\end{assumption}
\noindent The first assumption in (c) is satisfied in perforated domains, 
while the second one is satisfied in some defect domains, which does not include the standard 
analysis of high contrast domains. Analyzing high contrast domains would require different analysis 
(see \cite{Allaire}, for the periodic case and two-scale convergence approach).
\begin{example}
	We take the situation of Example \ref{primjer1} and the extension operator given there. 
	Additionally, we assume that the energy densities are pointwise coercive in part of the domain 
	$\Omega \backslash S^{h_n}$ with the constants $\varepsilon_n$, where $\varepsilon_n \to 0$. 
	The condition under which (c) from Assumption \ref{ass;prfff} is satisfied is that 
	$\varepsilon_n \gg r_n:=\sup \{r_{h_n}^i\}_{i\in \N}$ 
	(notice that high contrast appears when $\varepsilon_n \sim r_n$). Namely, 
	using \eqref{zadnjeprimjer}, we conclude that
	\begin{equation} \label{jelgotovo} 
	 \| \vect u -\vect u^{E^{h_n}} \|^2_{L^2}  
	 \leq Cr_n^2\|\sym \nabla_{h_n}  (\vect u -\vect u^{E^{h_n}})\|^2_{L^2(\Omega \backslash S^{h_n})}\,, 
	 \quad \forall \vect u \in H^1(\Omega, \R^3)\,.   
	\end{equation} 
Coercivity of the energy, Corollary \ref{kornthincor} from Appendix, property (c) of 
Assumption \ref{ass: perforated} and the above inequality imply 
\begin{align*}
\|\sym \nabla_{h_n} \vect u^{E^{h_n}}\|^2_{L^2}&+\varepsilon_n \|\sym \nabla_{h_n} \vect u\|^2_{L^2(\Omega \backslash S^{h_n})}  \leq C \left(\left\|\left(u^{h_n}_1,u^{h_n}_2, h_n u^{h_n}_3\right)\right\|_{L^2}+1 \right) \\ & \leq C\left(\left\|\left(u^{h_n}_1-u^{E^{h_n}}_1,u^{h_n}_2-u^{E^{h_n}}_2, h_n \left(u^{h_n}_3-u^{E^{h_n}}_3\right)\right)\right\|_{L^2}+\|(u^{E^{h_n}}_1,u^{E^{h_n}}_2, h_n u^{E^{h_n}}_3)\|_{L^2}+1 \right)\\
&  \leq C  \left(r_n\|\sym \nabla_{h_n}  (\vect u -\vect u^{E^{h_n}})\|_{L^2(\Omega \backslash S^{h_n})}+\|\sym \nabla_{h_n} \vect u^{E^{h_n}}\|_{L^2}+1 \right) \\
& \leq C  \left(r_n\|\sym \nabla_{h_n}  \vect u \|_{L^2(\Omega \backslash S^{h_n})}+\|\sym \nabla_{h_n} \vect u^{E^{h_n}}\|_{L^2}+1 \right).
\end{align*}
From the last inequality we conclude the boundedness of the sequences 
$(\|\sym \nabla_{h_n} \vect u^{E^{h_n}}\|_{L^2})_n $ and 
\\$(\varepsilon_n \|\sym \nabla_{h_n} \vect u\|^2_{L^2(\Omega \backslash S^{h_n})} )_n$. 
The first boundedness gives us the compactness result that we need 
(in the analysis of convergence of minimizers), while the second one, together with 
estimate \eqref{jelgotovo}, gives that
$\| \vect u -\vect u^{E^{h_n}} \|_{L^2} \to 0$, which is stronger than the requirement (c) in 
Assumption \ref{ass;prfff}. 
\end{example}

The analysis of convergence of minimizers follows standard arguments of the $\Gamma$-convergence.
To conclude compactness for sequence of minimizers, one has to take the sequence of minimizers 
$(\vect u_m^{h_n})_n$ for $h_n$-problem and look the sequence $(\vect u_m^{E^{h_n}})_n$, 
which has bounded scaled symmetric gradients. 
In the first case of (c), we have 
$\fenergy_{h_n} (\vect u_m^{h_n})=\fenergy_{h_n} (\vect u_m^{E^{h_n}})$ and the analysis proceeds 
in the same way as above. The term added to the energy is of the form
 $$-\int \theta (f_1 w_1+f_2 w_2+f_3 v)\dd x\,. $$ 
In the second case, first we easily conclude by testing with zero function 
that the sequence of minimizers satisfies the second condition in (c). 
In that way we have:
\begin{equation}
\label{igggor}
\left(u^{h_n}_{m,1},u^{h_n}_{m,2}, h_n u^{h_n}_{m,3}\right) 
- \left(u^{E^{h_n}}_{m,1},u^{E^{h_n}}_{m,2}, h_n u^{E^{h_n}}_{m,3}\right)\to 0 \textrm{ strongly in } L^2\,,
\end{equation}
\begin{equation*}
 \fenergy_{h_n} (\vect u^{h_n}_m) = \int_{\Omega} Q^{h_n} \left(x, \sym \nabla_{h_n}\vect u^{h_n}_m\right) \dd x 
 - \int_{\Omega} (f_1,f_2,f_3)\cdot \left(u^{E^{h_n}}_{m,1},u^{E^{h_n}}_{m,2}, h_n u^{E^{h_n}}_{m,3}\right) \dd x + o(1)\,,
\end{equation*}
where $\lim_{n \to \infty} o(1)=0$. Using (a) and (c) from 
Assumption \ref{ass: perforated}, we easily conclude that the sequence 
$\left(\vect u^{E^{h_n}}_m \right)_n$ has bounded symmetrized scaled gradients in the $L^2$-norm. 
To make the lower bound, one utilizes a simple observation: 
 \begin{equation*}
  \int_{\Omega} Q^{h_n} \left(x, \sym \nabla_{h_n}\vect u^{h_n}_m\right) \dd x = 
  \int_{\Omega} Q^{h_n} \left(x, \sym \nabla_{h_n}\vect u^{E^{h_n}}_m+\sym \nabla_{h_n}\left(\vect u^{h_n}_m-\vect u^{E^{h_n}}_m\right)\right) \dd x\,,
\end{equation*}
uses \eqref{igggor}, definition of the limit energy density $Q^0$ and the analysis 
from the previous section.
\begin{remark}  
In the analysis of bending rod or non-linear von K\'arm\'an plate, 
instead of Assumption \ref{ass;prfff}, the existence of another extension operator 
(besides the one from Assumption \ref{ass: perforated}), denoted by $\tilde{E}^{h_n}$, is required. 
Namely, to obtain the compactness result, one needs from the order of forces $h^2$ and $h^3$, 
conclude that the order of the term 
$\|\dist(\nabla_{h_n} \vect y_m^{\tilde{E}^{h_n}}, \SO 3)\|_{L^2(\Omega)}^2$ is $h^2$ and 
$h^4$, respectively (cf.~\cite{FJM06}). Here, $(\vect y_m^{\tilde{E}^{h_n}})_n$ denotes a sequence of 
extended minimizers. 
In the following, we will very briefly discuss the analysis of the bending perforated rod 
(we assume that $S^{h_n} \subseteq \Omega$ is filled with a material and the rest is empty).
A detailed description of the bending rod model is given for instance in \cite{MaVe14a},
and the main tool in obtaining the compactness result is the 
theorem on geometric rigidity \cite{FJM06}.
If the thin domain can be covered with cubes of size $h_n$, which can have only finite number of 
different geometrical shapes with perforations, one can, using the theorem on geometric rigidity, 
conclude that there exist a constant $C>0$ and for each $h_n$ an extension operator 
$\tilde{E}^{h_n} : H^1(S^{h_n},\R^3) \to H^1(\Omega,\R^3)$ satisfying:
\begin{eqnarray}
& & \label{borism0} \| \vect y^{\tilde{E}^{h_n}} \|_{H^1(\Omega)} 
\leq C \|  \vect y^{h_n} \|_{H^1(S^{h_n})}\,,\quad \forall \vect y ^{h_n} \in H^1(S^{h_n},\R^3)\,;\\
& & \label{borism1} \| \nabla_{h_n} \vect y^{\tilde{E}^{h_n}} \|_{L^2(\Omega)} 
\leq C \| \nabla_{h_n} \vect y^{h_n} \|_{L^2(S^{h_n})}\,,\quad 
\forall \vect y ^{h_n} \in H^1(S^{h_n},\R^3)\,;\\
 & & \label{borism2} \left\|\dist\left(\nabla_{h_n} \vect y^{\tilde{E}^{h_n}}, \SO 3\right) 
 \right \|_{L^2(\Omega)} 
 \leq  C \left\|\dist\left(\nabla_{h_n} \vect y^{h_n}, \SO 3\right) \right \|_{L^2(S^{h_n})}\,, 
 \quad \forall \vect y ^{h_n} \in H^1(S^{h_n},\R^3)\,.   
\end{eqnarray}
The extension mapping can be constructed on the unit cube as the one that preserves affine mappings. 
Applying the theorem on geometric rigidity, we then conclude (\ref{borism2}). 
If one takes the forces of the form $h^2 \vect f$, it is not difficult to conclude that  the 
sequence of minimizers satisfies 
\begin{equation*}
\limsup_{n\to\infty}\left\|\dist (\nabla_{h_n} \vect y_m^{\tilde{E}^{h_n}}, \SO 3)\right\|_{L^2(\Omega)}^2
 \dd x \leq Ch^2\,,
\end{equation*} 
for some $C>0$.
The analysis then continues in the standard way, by analyzing the sequence $(\vect y_m^{\tilde{E}^{h_n}})_n$. 
\end{remark}	
\subsection{Auxiliary claims}
Results of this subsection, will be used in establishing the locality property of the $\Gamma$-closure. The arguments follow the standard arguments of $\Gamma$-convergence. As in the next section we will assume the pointwise coerciveness assumption from \eqref{2.eq:A1}.

\begin{lemma}\label{3.14} 
Let $D\subset\R^2$ be open, bounded set having Lipschitz boundary and $(h_n)_n$ satisfying Assumption \ref{2:ass}.
Let $Q^{h_n}$ be a sequence of energy densities with the limit energy density $Q^0$. Then for every 
$\mat M \in \set S(\omega)$
\begin{align*}
\lim_{n\to\infty}\min_{\substack{\vect\psi\in H^1(D\times I,\R^3)\\ 
			\vect\psi=0\text{ on }\pa D\times I}}
			& \int_{D\times I}Q^{h_{n}}(x, \imath(\mat M) + \nabla_{h_{n}}\vect\psi )\dd x \nonumber\\
	& = \min_{\substack{\vect w\in H_0^1(D,\R^2)\,, \\ v\in H^2_0(D)}}
	\int_D Q^0(x',\mat M_1 + \sym \nabla\vect w, \mat M_2 - \nabla^2 v)\dd x'\,.
\end{align*}
Furthermore, for every $r>0$ we have
\begin{align*}
\liminf_{n\to\infty}\min_{\substack{\vect\psi\in H^1(D\times I,\R^3) \\
			\vect\psi=0\text{ on }\pa D\times I\,, 
			\\ \|(\psi_1,\psi_2,h_n\psi_3)\|_{L^2}\leq r}}
			& \int_{D\times I}Q^{h_{n}}(x, \imath(\mat M) + \nabla_{h_{n}}\vect\psi )\dd x \\ \nonumber
	& \geq \min_{\substack{\vect w\in H_0^1(D,\R^2)\,,\, v\in H^2_0(D)\,, \\ 
			\|\vect w\|_{L^2}^2 + \|v\|_{L^2}^2 \leq r^2 }}
	\int_D Q^0(x',\mat M_1 + \sym \nabla\vect w, \mat M_2 - \nabla^2 v)\dd x'\,.
\end{align*}
\end{lemma}
\begin{proof}
We prove the statements by means of the $\Gamma$-convergence. 
For a fixed $\mat M\in\set S(\omega)$, take a minimizing sequence $(\vect\psi^{h_n})_n\subset H^1(D\times I,\R^3)$ with 
$\vect\psi^{h_n} = 0$ on $\pa D\times I$ satisfying
\begin{equation*}
\int_{D\times I}Q^{h_{n}}(x, \imath(\mat M) + \nabla_{h_{n}}\vect\psi^{h_n} )\dd x = 
\min_{\substack{\vect\psi\in H^1(D\times I,\R^3) \\
		\vect\psi=0\text{ on }\pa D\times I }} 
	\int_{D\times I}Q^{h_{n}}(x, \imath(\mat M) + \nabla_{h_{n}}\vect\psi )\dd x\,.
\end{equation*}
Comparing with the zero function and using equi-boundednes and equi-coercivity of $Q^{h_n}$, we find the uniform bound
\begin{equation*}
\|\sym \nabla_{h_n}\vect \psi^{h_n}\|_{L^2}\leq C\|\mat M\|_{L^2}\,,\quad n\in\N\,.
\end{equation*}
Applying Lemma \ref{app:lem.limsup}, there exist $\vect w\in H_0^1(D,\R^2)$, $v\in H_0^2(D)$, and 
sequences $(\varphi^{h_n})_n\subset H_0^2(D)$, $(\bar{\vect\psi}^{h_n})_n\subset H^1(D\times I,\R^3)$
satisfying 
\begin{equation*}
\sym \nabla_{h_n}\vect\psi^{h_n} = \imath(\sym \nabla'\vect w - x_3\nabla'^2v) + \sym \nabla_{h_n}\bar{\vect\psi}^{h_n}
\end{equation*}
and $(\bar{\psi}_1^{h_n}, \bar{\psi}_2^{h_n}, h_n \bar{\psi}_3^{h_n}) \to 0$ in the $L^2$-norm as $n\to\infty$.
From the definition of $Q^0$ we obtain the lower bound
\begin{equation*}
\liminf_{n\to\infty}\min_{\substack{\vect\psi\in H^1(D\times I,\R^3) \\
	\vect\psi=0\text{ on }\pa D\times I }} 
	\int_{D\times I}Q^{h_{n}}(x, \imath(\mat M) + \nabla_{h_{n}}\vect\psi )\dd x
\geq \min_{\substack{\vect w\in H_0^1(D,\R^2)\,, \\ v\in H^2_0(D)}}
	\int_D Q^0(x',\mat M_1 + \sym \nabla'\vect w, \mat M_2 - \nabla'^2 v)\dd x'\,.
\end{equation*}
To prove the upper bound, take $\vect w_0\in H_0^1(D,\R^2)$ and $v_0\in H_0^2(D)$ such that
\begin{equation*}
\int_D Q^0(x',\mat M_1 + \sym \nabla'\vect w_0, \mat M_2 - \nabla'^2 v_0)\dd x' = 
\min_{\substack{\vect w\in H_0^1(D,\R^2)\,, \\ v\in H^2_0(D)}}
	\int_D Q^0(x',\mat M_1 + \sym \nabla'\vect w, \mat M_2 - \nabla'^2 v)\dd x'\,.
\end{equation*} 
Using Lemma \ref{lemma:equi-int1} (cf.~Remark \ref{rempromjena}), there exists a subsequence of $(h_n)_n$ (in the same notation)
and a sequence $(\vect\psi^{h_n})_n\subset H^1(D\times I)$ such that $\vect{\psi}^{h_n} = 0$ on $\pa D\times I$ and
\begin{align*}
\int_D Q^0(x',\mat M_1 + & \sym \nabla'\vect w_0, \mat M_2 - \nabla'^2 v_0)\dd x' \\
&= \lim_{n\to\infty}\int_{D\times I} Q^{h_n}(x,\imath(\mat M + \sym \nabla'\vect w_0 - x_3 \nabla'^2 v_0) +  \nabla_{h_n}\vect\psi^{h_n})\dd x
\end{align*}
Defining
\begin{equation*}
\vect{l}^{h_n} = \left(\begin{array}{c} \vect w_0 \\ \dfrac{v_0}{h_n}\end{array}\right)
 - x_3 \left(\begin{array}{c} \nabla'v_0 \\ 0 \end{array}\right)\,,\quad n\in\N\,,
\end{equation*}
and noticing that $\vect l^{h_n} = 0$ on $\pa D\times I$ and
\begin{equation*}
\sym\nabla_{h_n}\vect l^{h_n} = \imath(\sym\nabla'\vect w_0 - x_3\nabla'^2v_0)\,,
\end{equation*}
we conclude
\begin{align*}
\limsup_{n\to\infty} & \min_{\substack{\vect\psi\in H^1(D\times I,\R^3) \\
	\vect\psi=0\text{ on }\pa D\times I }} 
	\int_{D\times I}Q^{h_{n}}(x, \imath(\mat M) + \nabla_{h_{n}}\vect\psi )\dd x \\
& \leq \limsup_{n\to\infty}	\int_{D\times I}Q^{h_{n}}(x, \imath(\mat M) + \nabla_{h_{n}}(\vect l^{h_n} + \vect\psi^{h_n}) )\dd x\\
& = \int_D Q^0(x',\mat M_1 + \sym \nabla'\vect w_0, \mat M_2 - \nabla'^2 v_0)\dd x'\,,
\end{align*}
which establishes the upper bound.
The second statement follows the same way as the above proof of the lower bound.
\end{proof}

The statement of the following lemma is well known. Briefly, it states that the pointwise convergence of quadratic forms implies the convergence of minimizers.
\begin{lemma}\label{3.lem:minimizers}
Let $D\subset \R^2$ be open, bounded set having Lipschitz boundary and let $Q^n,Q: D\times \R_{\sym}^{2\times 2}\times \R_{\sym}^{2\times 2}\to \R$, $n\in\N$, 
be uniformly bounded and coercive quadratic forms satisfying 
$\lim_{n\to\infty}Q^n(x',\mat M_1,\mat M_2) = Q(x',\mat M_1,\mat M_2)$ for a.e.~$x'\in D$ and for all
$\mat M_1,\mat M_2\in\R_{\sym}^{2\times 2}$. Then for arbitrary 
$\mat M_1,\mat M_2\in L^2(D,\R_{\sym}^{2\times 2})$
\begin{align*}
\lim_{n\to\infty}\min_{\substack{\vect u\in H^1(D,\R^2)\,,\, v\in H^2(D)\,, \\ 
			(\vect u,v)\in \set N(D)}}
	&\int_D Q^n(x',\mat M_1 + \sym \nabla\vect u, \mat M_2 - \nabla^2 v)\dd x' \\
	& = \min_{\substack{\vect u\in H^1(D,\R^2)\,,\, v\in H^2(D)\,, \\ 
			(\vect u,v)\in \set N(D)}}
	\int_D Q(x',\mat M_1 + \sym \nabla\vect u, \mat M_2 - \nabla^2 v)\dd x'\,,
\end{align*}
where $\set N(D)$ is a closed subset in weak $H^1(D,\R^2)\times H^2(D)$ topology satisfying 
\begin{equation*}
(\vect u,v)\in\set N(D) \Rightarrow \|\vect u\|_{H^1} + \|v\|_{H^2} \leq C(\|\sym \vect u\|_{L^2} + \|v\|_{L^2})\,.
\end{equation*}
Moreover, if $(\vect u_n, v_n)_n$ are minimizers of 
\begin{equation*}
\min_{\substack{\vect u\in H^1(D,\R^2)\,,\, v\in H^2(D)\,, \\ 
			(\vect u,v)\in \set N(D)}}
	\int_D Q^n(x',\mat M_1 + \sym \nabla\vect u, \mat M_2 - \nabla^2 v)\dd x'\,,
\end{equation*}
then we have $\|\sym\nabla(\vect u - \tilde{\vect{u}})\|_{L^2}\to 0$ and $\|\nabla^2(v_n-\tilde v)\|_{L^2}\to 0$ as 
$n\to\infty$, where $(\tilde{\vect u},v)$ are minimizers of 
\begin{equation*}
\min_{\substack{\vect u\in H^1(D,\R^2)\,,\, v\in H^2(D)\,, \\ 
			(\vect u,v)\in \set N(D)}}
	\int_D Q(x',\mat M_1 + \sym \nabla\vect u, \mat M_2 - \nabla^2 v)\dd x'\,.
\end{equation*}
\end{lemma}
\begin{proof}
For the sake of brevity, assume $\mat M_1 = \mat M_2 = 0$. By the direct methods of the calculus of variations,
using the uniform coercivity and the Korn's inequality, minima $(\vect u_n, v_n)\in\set N(D)$ always exist:
\begin{equation*}
\int_D Q^n(x', \sym \nabla\vect u_n, - \nabla^2 v_n)\dd x' = \min_{\substack{\vect u\in H^1(D,\R^2)\,,\, v\in H^2(D)\,, \\ 
			(\vect u,v)\in \set N(D)}}\int_D Q^n(x',\sym \nabla\vect u,- \nabla^2 v)\dd x'\,.
\end{equation*} 
Let $\tilde{\vect u}\in H^1(D,\R^2)$ and $\tilde{v}\in H^2(D)$ be weak limits of $(\vect{u}_n)_n$ and $(v_n)_n$, 
respectively.
For $x'\in D$, let $\tensor{A}^n(x'): \R_{\sym}^{2\times 2}\times \R_{\sym}^{2\times 2}\to 
\R_{\sym}^{2\times 2}\times \R_{\sym}^{2\times 2}$ be the corresponding linear symmetric operator 
of the form $Q^n(x',\cdot,\cdot)$ satisfying
\begin{equation*}
Q^n(x',\mat M_1,\mat M_2) = \tensor{A}^n(x')(\mat M_1,\mat M_2) : (\mat M_1,\mat M_2)\,,
\quad\forall \mat M_1,\mat M_2\in \R_{\sym}^{2\times 2}\,,
\end{equation*}
and $\tensor{A}^n(x')$ be the corresponding operator of the form $Q(x',\cdot,\cdot)$. Note that $\tensor{A}^n(x')$, 
 $\tensor{A}(x')$ are bounded and satisfy $\tensor{A}^n(x')\to \tensor{A}(x')$ for a.e.~$x'\in D$. Therefore
\begin{align}\label{bukal3}
\int_D Q^n(x', & \sym \nabla\vect u_n, \nabla^2 v_n)\dd x'  = 
\int_D \tensor{A}^n(x')(\sym \nabla\tilde{\vect u}, \nabla^2 \tilde v) : (\sym \nabla\tilde{\vect u}, \nabla^2 \tilde v)\dd x'\\ \nonumber
& \quad + 2\int_D \tensor{A}^n(x')(\sym \nabla\tilde{\vect u}, \nabla^2 \tilde v)
 : (\sym \nabla(\vect u_n - \tilde{\vect u}), \nabla^2(v_n - \tilde v))\dd x'\\ \nonumber
& \quad + \int_D \tensor{A}^n(x')(\sym \nabla(\vect u_n - \tilde{\vect u}), \nabla^2(v_n - \tilde v))
 : (\sym \nabla(\vect u_n - \tilde{\vect u}), \nabla^2(v_n - \tilde v))\dd x'\\ \nonumber
& \geq \int_D \tensor{A}^n(x')(\sym \nabla\tilde{\vect u}, \nabla^2\tilde v) : (\sym \nabla\tilde{\vect u}, \nabla^2\tilde v)\dd x'\\ \nonumber
& \quad + 2\int_D \tensor{A}^n(x')(\sym \nabla\tilde{\vect u}, \nabla^2 \tilde v)
 : (\sym \nabla(\vect u_n - \tilde{\vect u}), \nabla^2(v_n - \tilde v))\dd x'\\ \nonumber
& \quad + \frac{\alpha}{12}(\|\sym \nabla(\vect u_n - \tilde{\vect u})\|_{L^2}^2 + \|\nabla^2(v_n - \tilde v)\|_{L^2}^2)\,.
\end{align}
By the Lebesgue dominated convergence theorem 
\begin{equation*}
\int_D \tensor{A}^n(x')(\sym \nabla\tilde{\vect u}, \nabla^2 \tilde v) : (\sym \nabla\tilde{\vect u}, \nabla^2 \tilde v)\dd x'
\to \int_D \tensor{A}(x')(\sym \nabla\tilde{\vect u}, \nabla^2 \tilde v) : (\sym \nabla\tilde{\vect u}, \nabla^2 \tilde v)\dd x'\,,
\end{equation*}
as $n\to\infty$.
Using the strong convergence of operators $\tensor{A}^n(x')\to\tensor{A}(x')$ and the weak convergence of the sequences
$(\vect{u}_n)_n$ and $(v_n)_n$, we conclude 
\begin{equation*}
\int_D \tensor{A}^n(x')(\sym \nabla\tilde{\vect u}, \nabla^2 \tilde v)
 : (\sym \nabla(\vect u_n - \tilde{\vect u}), \nabla^2(v_n - \tilde v))\dd x' \to 0\,.
\end{equation*}
Thus, 
\begin{equation*}
\liminf_{n\to\infty}\int_D Q^n(x', \sym \nabla\vect u_n, \nabla^2 v_n)\dd x' \geq 
\int_D Q(x', \sym \nabla\tilde{\vect u}, \nabla^2 \tilde v)\dd x'\,,
\end{equation*}
which shows the lower bound. The upper bound is trivial by simply using the constant sequences 
$\vect{u}_n = \tilde{\vect u}$ and $v_n = \tilde{v}$. Along the same path we also proved that 
$(\tilde{\vect u}, \tilde v)$ are indeed minimizers of
\begin{equation*}
\int_D Q(x', \sym \nabla\vect u, \nabla^2 v)\dd x'\,.
\end{equation*}
The strong convergence follows by going back to the inequality \eqref{bukal3} and letting $n \to \infty$. 
\end{proof}
The following result is analogous to \cite[Lemma 2.3]{BaBa09}. 
\begin{lemma}\label{3.cor:minimizers}
Let $Q:\omega \times \R_{\sym}^{2\times 2}\times \R_{\sym}^{2\times 2}\to \R$ be uniformly bounded and coercive quadratic functional and
$(s_m)_m$ decreasing to zero sequence. Then for a.e.~$x_0'\in\omega$ and arbitrary $D\subset[0,1]^2$, 
$\mat M_1,\mat M_2\in L^2(D,\R_{\sym}^{2\times 2})$,
\begin{align*}
\lim_{m\to\infty}\min_{\substack{\vect u\in H^1(D,\R^2)\,,\, v\in H^2(D)\,, \\ 
			(\vect u,v)\in \set N(D)}}
	&\int_D Q(x_0' + s_m z',\mat M_1 + \sym \nabla\vect u, \mat M_2 - \nabla^2 v)\dd z' \\
	& = \min_{\substack{\vect u\in H^1(D,\R^2)\,,\, v\in H^2(D)\,, \\ 
			(\vect u,v)\in \set N(D)}}
	\int_D Q(x_0',\mat M_1 + \sym \nabla\vect u, \mat M_2 - \nabla^2 v)\dd z'\,,
\end{align*}
where $\set N(D)$ is from Lemma \ref{3.lem:minimizers} above.
\end{lemma}

\begin{proof}
Using Lemma 5.38 from \cite{AFP00}, for a.e.~$x_0'\in\omega$, there exists a subsequence, still denoted by $(s_m)_m$,
and $E\subset [0,1]^2$ of measure zero such that
\begin{equation*}
\lim_{m\to\infty}Q(x_0' + s_mz',\cdot,\cdot) = Q(x_0',\cdot,\cdot)
\end{equation*}
locally uniformly in $\R_{\sym}^{2\times 2}\times \R_{\sym}^{2\times 2}$ for every $z'\in [0,1]^2\backslash E$. 
The statement then follows from Lemma \ref{3.lem:minimizers}.
\end{proof}

\section{Locality of $\Gamma$-closure}\label{sec:4}  

In this section we prove Theorem \ref{thm:2}, which states that every effective energy density obtained by 
mixing $N$ different materials can be locally (in almost every point) recovered as pointwise 
limit of in-plane periodically homogenized energy densities.  We assume the pointwise uniform boundedness and coercivity condition \eqref{2.eq:A1}. 

\noindent First we recall the definition of variational functinals, which will be frequently used in the sequel. 

\begin{remark}
Utilizing topological characterizations and Lemma \ref{lemma:equi-int1} we have the following identities for every $A \subset \omega$ with Lipschitz boundary (see Remark \ref{2.rem:top}(d))
\begin{align*}
\Kchi (\mat M,A) 
	& = \sup_{\set{U}\subset\set{N}(0)}\liminf_{n\to\infty}\widetilde{\set K}_{\chi^{h_n}}(\mat M, A,\set U)
	  = \sup_{\set{U}\subset\set{N}(0)}\limsup_{n\to\infty}\widetilde{\set K}_{\chi^{h_n}}(\mat M, A,\set U)\,,\\
	& = \sup_{\set{U}\subset\set{N}(0)}\liminf_{n\to\infty}\widetilde{\set K}_{\chi^{h_n}}^0(\mat M, A,\set U)
	  = \sup_{\set{U}\subset\set{N}(0)}\limsup_{n\to\infty}\widetilde{\set K}_{\chi^{h_n}}^0(\mat M, A,\set U)\,,
\end{align*}
where $\widetilde{\set K}_{\chi^{h_n}}$ is as in Remark \ref{2.rem:top}(a) and  
\begin{equation}\label{4.def:Kchi0}
\widetilde{\set K}_{\chi^{h_n}}^0(\mat M, A,\set U) = \inf_{\substack{\vect{\psi}\in H^1(A \times I,\R^3)\\ 
 				(\psi_1,\psi_2,h_n\psi_3)\in\set{U} \\ \vect\psi = 0\text{ on }\pa A\times I}}\int_{A\times I}
 				Q^{\chi^{h_n}}(x,\imath(\mat M) + \sym\nabla_{h_n}\vect\psi)\dd x\,.
\end{equation}
\end{remark}
\noindent The following result, given in \cite[Lemma 3.2]{BaBa09}, states that every effective energy density can be
generated by a sequence of mixtures of fixed volume portions. 
\begin{lemma}\label{4.lem:fvfrac}
Let $(\chi^{h_n})_n\subset\set X^N(\Omega)$ be a sequence of mixtures with limit energy density $Q$ and
assume that $\chi^{h_n} \overset{*}{\rightharpoonup}\mu\in L^\infty(\Omega,[0,1]^N)$. 
There exists another sequence
$(\tilde{\chi}^{h_n})_n\subset\set X^N(\Omega)$ such that $\tilde\chi^{h_n} \overset{*}{\rightharpoonup}\mu$
in $L^\infty(\Omega,[0,1]^N)$ and 
\begin{equation*}
\int_\Omega\tilde{\chi}_i^{h_n}(x)\dd x = \bar{\mu}_i\,,\quad \forall n\in\N\,,\ i=1,\ldots,N\,,
\end{equation*}
where $\bar{\mu}_i = \int_\Omega \mu_i(x)\dd x$. Moreover, the sequence of mixtures with fixed volume
fractions $(\tilde{\chi}^{h_n})_n$ has the same limit energy density $Q$.
\end{lemma}
\begin{proof}
Proof of the existence of $(\tilde{\chi}^{h_n})_n$ follows \cite[Lemma 3.2]{BaBa09}. 
Construction is performed in such a way that $|\{\chi^{h_n} \neq \tilde{\chi}^{h_n} \}| \to 0$. 
In order to prove that the limit energy density is the same, it is enough to see that for every 
sequence $(\vect \psi^{h_n})_n \subset H^1(\Omega, \R^3)$, such that 
$(|\sym\nabla_{h_{n}}\vect{\psi}^{h_n}|^2)_n$ is equi-integrable, we have
\begin{eqnarray*}
&  & \left|\int_{\Omega} Q^{\chi^{h_n}}(x,\imath(\mat M) + \sym\nabla_{h_n}\vect\psi^{h_n})\,\dd x 
- \int_{\Omega}Q^{\tilde{\chi}^{h_n}}(x,\imath(\mat M) + \sym\nabla_{h_n}\vect\psi^{h_n})\,\dd x \right|  
\leq \\ & & \hspace{+35ex} 2 \beta  \int_{\{ \tilde{\chi}^{h_n} \neq \chi^{h_n}\} }|\iota(\mat M) + 
\sym\nabla_{h_{n}}\vect{\psi}^{h_n}|^2 \dd x  \to 0\,,\qquad \forall \mat M \in \set S(\omega)\,. 
\end{eqnarray*}
\end{proof}
\begin{definition}
For a sequence $(h_n)_n$ monotonically decreasing to zero, $\theta \in[0,1]^N$ 
such that $\sum_{i=1}^N\theta_i=1$ and $\Omega = \omega\times I$ with $\omega\subset\R^2$, 
we define $\set G_{\theta}^{(h_n)}(\omega)$ as the set of all homogeneous quadratic forms
$Q:\R_{\sym}^{2\times2}\times\R_{\sym}^{2\times2}\to\R$ which are limit energy densities of a
mixture $(\chi^{h_n})_n\subset\set X^N(\Omega)$ satisfying:  
$\int_\Omega \chi^{h_n}_i(x)\dd x  = \theta_i$ and 
$\chi^{h_n}\overset{*}{\rightharpoonup}{\mu}$
in $L^\infty(\Omega,[0,1]^N)$ for some $\mu \in L^\infty(I,[0,1]^N)$ such that 
$\int_I\mu_i(x_3)\dd x_3 = \theta_i$, for all $i=1,\ldots,N$.
\end{definition} 

\begin{remark}
Note that $\set G_{\theta}^{(h_n)}(\omega)$ a priori depends on the sequence $(h_n)_n$, and at this point it is not clear 
whether it is independent of $(h_n)$. We will prove that in the sequel. 
\end{remark}

\begin{proposition}\label{4.prop:dilatation}
Let $\Omega = \omega\times I$ and $\tilde\Omega = \tilde\omega\times I$, where $\omega,\,\tilde\omega\subset\R^2$ open, bounded set having Lipschitz boundary
and $s>0$ such that
$x_0'+s\tilde{\omega}\subset\omega$ for some $x_0'\in\omega$. For any $(h_n)_n$, $h_n\downarrow0$, 
and $\theta\in [0,1]^N$ we have
\begin{equation}
\set G_{\theta}^{(h_n)}(\omega) \subseteq \set G_{\theta}^{(h_n/s)}(\tilde\omega)\,.
\end{equation}
\end{proposition}
\begin{proof}
Let $Q\in \set G_{\theta}^{(h_n)}(\omega)$ be a homogeneous limit energy density of the sequence 
of mixtures $(\chi^{h_n})_n$ of fixed volume fractions $\theta_i$ satisfying 
$\chi^{h_n}_i\overset{*}{\rightharpoonup}\mu_i$ for some $\mu \in L^\infty(I,[0,1]^N)$ 
such that $\int_I\mu_i(x_3)\dd x_3 = \theta_i$, for $i=1,\ldots,N$. 
Define a new sequence of mixtures $(\tilde{\chi}^{h_n'})_n\subset\set X^N(\tilde{\Omega})$, 
depending on the scaled sequence of thickness $h_n' = h_n/s$, $n\in\N$, by
\begin{equation*}
\tilde{\chi}^{h_n'}(y',y_3) = \chi^{h_n}(x_0' + sy',y_3)\,,\quad y\in\tilde{\Omega}\,.
\end{equation*}
For arbitrary $\phi\in L^1(\tilde{\Omega})$, one easily concludes that 
$\int_{\tilde{\Omega}}\tilde{\chi}_i^{h_n'}(y)\phi(y)\dd y \to \int_{\tilde{\Omega}}\mu_i(y_3)\phi(y)\dd y$,
as $n\to\infty$, hence $\tilde\chi^{h_n'}\overset{*}{\rightharpoonup}{\mu}$ 
in $L^\infty(\tilde\Omega,[0,1]^N)$.
Using the homogeneity of $Q$ and the definition of $\Kchi$, for arbitrary 
$\mat M_1,\mat M_2\in \R^{2\times2}_{\sym}$, we have
\begin{align*}
Q(\mat M_1,\mat M_2) &= \frac{1}{s^2|\tilde\omega|}\Kchit(\mat M_1 + x_3\mat M_2, x_0' + s\tilde{\omega})\\
& = \frac{1}{s^2|\tilde\omega|}\lim_{r\to0}\liminf_{n\to\infty}
\inf_{\substack{\vect \psi\in H^1((x_0'+s\tilde\omega)\times I,\R^3),\\ (\psi_1,\psi_2,h_n\psi_3)\in B(r,r)}} 
\int_{(x_0'+s\tilde\omega)\times I}Q^{\chi^{h_n}}(x, \imath(\mat M) + \nabla_{h_n}\vect\psi(x))\dd x \\
&= \frac{1}{s^2|\tilde\omega|}\lim_{r\to0}\limsup_{n\to\infty}
\inf_{\substack{\vect \psi\in H^1((x_0'+s\tilde\omega)\times I,\R^3),\\ (\psi_1,\psi_2,h_n\psi_3)\in B(r,r)}} 
\int_{(x_0'+s\tilde\omega)\times I}Q^{\chi^{h_n}}(x, \imath(\mat M) + \nabla_{h_n}\vect\psi(x))\dd x\,. 
\end{align*} 
Next, for $\vect\psi\in H^1(\Omega,\R^3)$ define $\tilde{\vect{\psi}}(y',y_3) = \frac{1}{s}\vect{\psi}(x_0'+sy',y_3)$ and
note that $\nabla_{h_n}\vect\psi = \nabla_{h_n'}\tilde{\vect{\psi}}$.
Changing variables in the above integral yields
\begin{align*}
Q(\mat M_1,\mat M_2) &= \frac{1}{|\tilde\omega|}\lim_{r\to 0}\liminf_{n\to\infty}
\inf_{\substack{\tilde{\vect\psi}\in H^1(\tilde\omega\times I,\R^3),\\ 
(\tilde\psi_1,\tilde\psi_2,h_n'\tilde\psi_3)\in B(r/s,r/s^2)}} 
\int_{\tilde\omega\times I}Q^{\tilde\chi^{h_n'}}(y, \imath(\mat M) + \nabla_{h_n'}\tilde{\vect{\psi}}(y))\dd y\\
&= \frac{1}{|\tilde\omega|}\lim_{r\to0}\limsup_{n\to\infty}
\inf_{\substack{\tilde{\vect\psi}\in H^1(\tilde\omega\times I,\R^3),\\ 
		(\tilde\psi_1,\tilde\psi_2,h_n'\tilde\psi_3)\in B(r/s,r/s^2)}} 
\int_{\tilde\omega\times I}Q^{\tilde\chi^{h_n'}}(y, \imath(\mat M) + \nabla_{h_n'}\tilde{\vect{\psi}}(y))\dd y  
\\
&= \frac{1}{|\tilde\omega|}\mathcal{K}_{(\tilde\chi^{h_n'})}(\mat M_1 + x_3\mat M_2, \tilde{\omega}).
\end{align*}
Hence, $(\tilde{\chi}^{h_n'})_n$ has the same limit energy density $Q$. 
Finally, applying Lemma \ref{4.lem:fvfrac} we infer that $Q\in \set{G}_{\theta}^{(h_n')}(\tilde\omega)$.
\end{proof}


\begin{proposition}\label{4.prop:closure}
Let $\Omega = \omega\times I$ with $\omega \subset\R^2$ open, bounded set having Lipschitz boundary, $(h_n)_n$ a sequence, $h_n\downarrow0$, 
and $\theta\in [0,1]^N$. 
Set $\set{G}_{\theta}^{(h_n)}(\omega)$ is closed.
\end{proposition}

\begin{proof}
Let $(Q_m)_m\subset \set{G}_{\theta}^{(h_n)}(\omega)$
be a sequence of homogeneous quadratic forms, $Q$ its limit, 
and $(\chi^{h_n,m})_n\subset \set X^N(\Omega)$ respective sequences of 
mixtures with limit energy densities $Q_m$.
Take an arbitrary $\mat M_1,\mat M_2\in \R^{2\times2}_{\sym}$ and $A\subset\omega$ 
open subset with Lipschitz boundary. By homogeneity of $Q_m$ and definition of 
$\widetilde{\mathcal{K}}_{\chi^{h_n,m}}^0$ (cf.~\ref{4.def:Kchi0}), for arbitrary $r>0$ we conclude 

\begin{equation*}
\limsup_{n\to\infty}\widetilde{\mathcal{K}}_{\chi^{h_n,m}}^0(\mat M_1 + x_3\mat M_2, A, B(r)) \leq 
|A|Q_m(\mat M_1,\mat M_2)\,,
\end{equation*}
where $B(r)$ denotes the ball of radius $r>0$ around the origin in the strong $L^2$-topology.
Therefore,
\begin{equation}\label{4.eq:limlimsup}
\limsup_{m\to\infty}\limsup_{n\to\infty}\widetilde{\mathcal{K}}_{\chi^{h_n,m}}^0(\mat M_1 + x_3\mat M_2, A, B(r)) \leq
|A|Q(\mat M_1,\mat M_2)\,.
\end{equation}
Applying Lemma \ref{3.14} and using homogeneity and convexity of quadratic forms $Q_m$, for arbitrary $r>0$ we
obtain
\begin{align*}
\liminf_{n\to\infty}\widetilde{\mathcal{K}}_{\chi^{h_n,m}}^0(\mat M_1 + x_3\mat M_2, A, B(r)) & \geq 
\min_{\substack{\vect w\in H_0^1(A,\R^2), \\ v\in H_0^2(A)}}
\int_A Q_m(\mat M_1 + \sym\nabla'\vect w,\mat M_2 - \nabla'^2v)\dd x'\\
&\geq |A|Q_m(\mat M_1,\mat M_2)\,,
\end{align*}
which implies
\begin{equation}\label{4.eq:limliminf}
\liminf_{m\to\infty}\liminf_{n\to\infty}\widetilde{\mathcal{K}}_{\chi^{h_n,m}}^0(\mat M_1 + x_3\mat M_2, A, B(r)) \geq
|A|Q(\mat M_1,\mat M_2)\,.
\end{equation}


By assumption there exists a sequence $(\mu^m)_m\subset L^\infty(I,[0,1]^N)$ such that
$\chi_i^{h_n,m}\overset{*}{\rightharpoonup}\mu_i^m$ for every $i=1,\ldots,N$ and $m\in\N$.
According to the Banach-Alaoglu theorem, there exists a subsequence still denoted by $(\mu^m)_m$ which 
weakly* converges to some $\mu\in L^\infty(I,[0,1]^N)$.
Next we take countable dense families: $(D_k)_{k}$ of open Lipschitz subsets in $\omega$; 
$(\mat M_{1,j},\mat M_{2,j})_j\subset \R_{\sym}^{2\times 2}\times \R_{\sym}^{2\times 2}$;
$(\phi_k)_k\subset L^1(\Omega)$; monotonically decreasing to zero sequence $(r_l)_l$; and
define doubly indexed family $(g_{n,m})_{n,m\in\N}$ by
\begin{align*}
g_{n,m} & = \sum_{k=1}^\infty 2^{-k}\min\left\{\max_{i=1,\ldots,N}\left|
\int_{\Omega}\left (\chi_i^{h_n,m}(x)-\mu_i^m(x_3)\right)\phi_k(x)\dd x\right|\,, 1\right\}\\
& \quad + \sum_{j,k,l=1}^\infty 2^{-j-k-l}\min\left\{\left|\frac{1}{|D_k|}
\widetilde{\mathcal{K}}_{\chi^{h_n,m}}^0(\mat M_{1,j} + x_3\mat M_{2,j}, D_k, B(r_l)) - 
Q(\mat M_{1,j},\mat M_{2,j})\right|\,, 1\right\}\,.
\end{align*}
The weak* convergence $\chi_i^{h_n,m}\overset{*}{\rightharpoonup}\mu_i^m$ 
for every $i=1,\ldots,N$ and $m\in\N$,
and bounds (\ref{4.eq:limlimsup}) -- (\ref{4.eq:limliminf}) imply 
$\liminf_{m\to\infty}\liminf_{n\to\infty}g_{n,m} = \limsup_{m\to\infty}\limsup_{n\to\infty}g_{n,m} = 0$.  Utilizing 
the standard diagonalization procedure (cf.~\cite[Lemma 1.15 and Corollary 1.16]{Att84}), there exists an increasing
function $n\mapsto m(n)$ such that
\begin{equation*}
\lim_{n\to\infty}g_{n,m(n)} = 0\,. 
\end{equation*}
Therefore, using $\mu^m\overset{*}{\rightharpoonup}\mu$, we obtain
\begin{align}
&\lim_{n\to\infty}\int_{\Omega}\left(\chi_i^{h_n,m(n)}(x) -\mu_i(x_3)\right)\phi_k(x)\dd x = 0\,, 
\quad \forall k\in\N\,, \ i = 1,\ldots,N\,,\label{4.eq:closure_1}\\
&\lim_{n\to\infty}\left(\frac{1}{|D_k|}
\widetilde{\mathcal{K}}_{\chi^{h_n,m(n)}}^0(\mat M_{1,j} + x_3\mat M_{2,j}, D_k, B(r_l)) - 
Q(\mat M_{1,j},\mat M_{2,j})\right) = 0\,, \quad \forall j,k,l\in\N\,.\label{4.eq:closure_2}
\end{align}
It is important to note here that in the diagonalization procedure we do not need to pass on a subsequence 
with respect to $n$.
Defining $\tilde\chi^{h_n} := \chi^{h_n,m(n)}$, $n\in\N$, we conclude from (\ref{4.eq:closure_1}), 
using the density argument, that
$\tilde\chi^{h_n} \overset{*}{\rightharpoonup}\mu$ in $L^\infty(\Omega,[0,1]^N)$. 
Using (\ref{4.eq:closure_2}) and 
the fact that
\begin{equation*}
\lim_{l\to\infty}\lim_{n\to\infty}\widetilde{\mathcal{K}}_{\tilde\chi^{h_n}}^0(\mat M_{1,j} + x_3\mat M_{2,j}, D_k, B(r_l)) = 
\mathcal{K}_{(\tilde\chi^{h_n})}(\mat M_{1,j} + x_3\mat M_{2,j}, D_k)\,,
\end{equation*}
again by the density argument for sets we conclude that $Q$ is the limit energy density of the mixture 
$(\tilde\chi^{h_n})_n\subset \set{X}^N(\Omega)$.
\end{proof}


\begin{proposition}\label{4.prop:geqp}
For every $(h_n)_n$ monotonically decreasing to zero we have  $\set G_{\theta}^{(h_n)}([0,1]^2) = \set P_{\theta}$. 
\end{proposition}

\begin{proof}
Obviously by construction $\set{Q}_{\theta} \subset \set{G}_{\theta}^{(h_n)}([0,1]^2)$, 
and by the closure property we have
inclusion $\set{P}_{\theta} \subseteq \set{G}_{\theta}^{(h_n)}([0,1]^2)$.
On the other hand, let $(\chi^{h_n})_n\subset \set X^N([0,1]^2\times I)$ such that
$\int_{[0,1]^2\times I}\chi^{h_n}_i(z)\dd z = \theta_i$ 
for $i=1,\ldots,N$, $\chi^{h_n} \overset{*}{\rightharpoonup}\mu \in L^\infty(I,[0,1]^N)$, 
and $\chi^{h_n}$ has the homogeneous limit energy density $Q\in \set{G}_{\theta}^{(h_n)}([0,1]^2)$. 
Using the convexity of $Q$ and applying Lemma \ref{3.14}, for every $\mat M_1,\mat M_2\in\R_{\sym}^{2\times2}$ we obtain
\begin{align*} 
Q(\mat M_1,\mat M_2) & = 
\min_{\substack{\vect w\in H^1(\T,\R^2),\\ v\in H^2(\T) }}
\int_{\T} Q(\mat M_1 + \sym\nabla'\vect w,\mat M_2 - \nabla'^2v)\dd x'\\
& = \lim_{n\to\infty}\min_{\vect\psi\in H^1(\T\times I,\R^3)}\int_{\T\times I}
Q^{\chi^{h_n}}(x,\imath(\mat M_1 + x_3\mat M_2) + \nabla_{h_n}\vect\psi)\dd x \\
& = \lim_{n\to\infty}\min_{\vect\psi\in H^1(\T\times I,\R^3)}\sum_{i=1}^{N}\int_{A_i^{h_n}}
Q_i^{h_n}(x,\imath(\mat M_1 + x_3\mat M_2) + \nabla_{h_n}\vect\psi)\dd x\,,
\end{align*}
where for each $n\in\N$, $\{A_i^{h_n}\}_{i=1,\ldots,N}$ is the partition of $\T\times I$. 
Thus, $Q\in \set{P}_{\theta}$.
\end{proof}
\begin{remark}
It is not strange that we can approximate the limit energy density by densities of the form $Q_{\gamma}$ with arbitrary small $\gamma$'s 
(see \eqref{int:def.qhom}). Recall, the meaning of the parameter $\gamma$ is $\gamma=\lim_{h \to 0} (h/\varepsilon (h))$, where 
$\varepsilon(h)$ is the period of oscillations. Thus, by taking that period arbitrary large by multiplying it with a natural number, 
we can make $\gamma$ as small as we wish. 
\end{remark}
 
\begin{corollary}
$\set G_{\theta}^{(h_n)}(\omega)$ is independent of the sequence $(h_n)_n$ and the set $\omega$. 
\end{corollary}
\begin{proof}
According to Proposition \ref{4.prop:dilatation}, there exist
$s_1,s_2>0$, such that 
\begin{equation} \label{4.eq:h_nomega}
\set G_{\theta}^{(h_n/s_1)}([0,1]^2) \subseteq \set G_{\theta}^{(h_n)}(\omega) 
\subseteq \set G_{\theta}^{(h_n/s_2)}([0,1]^2)\,.
\end{equation}
Since, according to the previous proposition, $\set G_{\theta}^{(h_n)}([0,1]^2)$ does not depend
on the sequence $(h_n)_n$, we have equalities in (\ref{4.eq:h_nomega}) and the claim follows.
\end{proof}

The following lemma is helpful for proving Theorem \ref{thm:2}. 
We employ the following notation: Let $(\chi^{h_n})_n\subset \set X^N(\Omega)$ be a sequence
of mixtures, $x_0'\in\omega$ and $(s_m)_m$ decreasing to zero sequence of positive numbers such that
$x_0' + s_1[0,1]^2\subset\omega$. Then for each $s_m$, we denote by $\chi^{h_{n,m}}(x',x_3) := \chi^{h_n}(x_0' + s_mx',x_3)$ 
the sequence in $\set X^N([0,1]^2\times I)$, where $h_{n,m} = h_n/s_m$.
Also, we define  $B(r_1,r_2) = \{\vect\psi\in H^1(D\times I,\R^3)\ :\ \|(\psi_1,\psi_2)\|_{L^2}\leq r_1\,, \|\psi_3\|_{L^2}\leq r_2\}$, for some $D \subset [0,1]^2$.

\begin{lemma}
Let $(\chi^{h_n})_n\subset \set X^N(\Omega)$ has the limit energy density $Q$. Then for almost every $x_0'\in\omega$, 
every $D\subset [0,1]^2$ of class $C^{1,1}$, $r>0$, and $\mat M_1,\mat M_2\in \R^{2\times2}_{\sym}$ we have
\begin{align*}
Q(x_0',\mat M_1,\mat M_2) 
& = \liminf_{m\to\infty}\liminf_{n\to\infty}\frac{1}{|D|}\widetilde{\mathcal{K}}_{\chi^{h_{n,m}}}^0(\mat M_1 + x_3\mat M_2, D, B(r,r))\\
& = \limsup_{m\to\infty}\limsup_{n\to\infty}\frac{1}{|D|}\widetilde{\mathcal{K}}_{\chi^{h_{n,m}}}^0(\mat M_1 + x_3\mat M_2, D, B(r,r)).
\end{align*}

\end{lemma}
\begin{proof}
Let $\mat M_1,\mat M_2\in \R^{2\times2}_{\sym}$. 
Take the sequence of minima 
$(\vect\psi^{n,m})_n\subset H^1((x_0' + s_mD)\times I,\R^3)$ such that for every $m\in\N$
\begin{equation*}
\widetilde{\mathcal{K}}_{\chi^{h_n}}^0(\mat M_1 + x_3\mat M_2, x_0' + s_mD, B(r_1,r_2)) = 
\int_{(x_0' + s_mD)\times I}Q^{\chi^{h_n}}(x,\imath(\mat M_1 + x_3\mat M_2) + \nabla_{h_n}\vect{\psi}^{n,m})\dd x\,,
\end{equation*}
$\vect\psi^{n,m} = 0$ on $\pa(x_0'+s_mD)\times I$, $\|(\psi_1^{n,m},\psi_2^{n,m})\|_{L^2}\leq r_1$, and
$\|h_n\psi_3^{n,m}\|_{L^2}\leq r_2$.
For arbitrary $m\in\N$ and $\vect{\psi}\in H^1((x_0'+s_mD)\times I,\R^3)$, 
define $\tilde{\vect{\psi}}(z',z_3) = \frac{1}{s_m}\vect{\psi}(x_0' + s_mz', z_3)\in H^1(D\times I,\R^3)$,
and observe that
\begin{equation*}
\nabla_{h_{n,m}}\tilde{\vect{\psi}} = \nabla_{h_n}\vect{\psi}\,.
\end{equation*}
Next, changing variables in the above integral, we obtain
\begin{align*}
\widetilde{\set{K}}_{\chi^{h_n}}^0(\mat M_1 + x_3\mat M_2,x_0' + s_mD, B(r,r)) 
		& = s_m^2 \int_{D\times I}Q^{\chi^{h_{n,m}}}(x,\imath(\mat M_1 + x_3\mat M_2) 
		+ \nabla_{h_{n,m}}\tilde{\vect{\psi}}^{n,m})\dd x \\
		& = s_m^2 \mathcal{K}_{\chi^{h_{n,m}}}^0(\mat M_1 + x_3\mat M_2, D, B(r/s_m,r/s_m^2))\,.
\end{align*}
Therefore, by the definition of $\widetilde{\set{K}}_{\chi^{h_n}}^0$,
\begin{align*}
\limsup_{n\to\infty}\widetilde{\mathcal{K}}_{\chi^{h_{n,m}}}^0(\mat M_1 + x_3\mat M_2, D, B(r,r)) & =
\limsup_{n\to\infty}\frac{1}{s_m^2}\widetilde{\mathcal{K}}_{\chi^{h_{n}}}^0(\mat M_1 + x_3\mat M_2, x_0' + s_mD,B(s_mr,s_m^2r)) \\
& \leq \frac{1}{s_m^2}\Kchit(\mat M_1 + x_3\mat M_2, x_0' + s_mD) \\
& = \frac{1}{s_m^2}\int_{x_0' + s_mD}Q(x',\mat M_1,\mat M_2)\dd x' = \int_D Q(x_0' + s_mx', \mat M_1, \mat M_2)\dd x'\,.
\end{align*}
Finally, for any $x_0'$, Lebesgue point of $Q$, we get


\begin{equation*}
\limsup_{m\to\infty}\limsup_{n\to\infty}\widetilde{\mathcal{K}}_{\chi^{h_{n,m}}}^0(\mat M_1 + x_3\mat M_2, D, B(r,r)) 
\leq |D|Q(x_0',\mat M_1, \mat M_2)\,.
\end{equation*}
On the other hand, using Lemma \ref{3.14}, we compute the lower bound as
\begin{align*}
\liminf_{n\to\infty}\widetilde{\mathcal{K}}_{\chi^{h_{n,m}}}^0 & (\mat M_1 + x_3\mat M_2, D, B(r,r)) =
\liminf_{n\to\infty}\frac{1}{s_m^2}\widetilde{\mathcal{K}}_{\chi^{h_{n}}}^0(\mat M_1 + x_3\mat M_2, x_0' + s_mD,B(s_mr,s_m^2r))\\
& \geq \liminf_{n\to\infty}\frac{1}{s_m^2}\widetilde{\mathcal{K}}_{\chi^{h_{n}}}^0(\mat M_1 + x_3\mat M_2, x_0' + s_mD,B(\infty,\infty))\\
& \geq \min_{\substack{\vect w\in H_0^1(x_0' + s_mD,\R^2), \\ v\in H_0^2(x_0' + s_mD) 
}}
\frac{1}{s_m^2} \int_{x_0' + s_mD} Q(x', \mat M_1 + \sym\nabla'\vect w,\mat M_2 - \nabla'^2v)\dd x'\\
& = \min_{\substack{\vect w\in H_0^1(D,\R^2), \\ v\in H_0^2(D) }}
\int_{D} Q(x_0' + s_mx', \mat M_1 + \sym\nabla'\vect w,\mat M_2 - \nabla'^2v)\dd x'\,.
\end{align*}
Applying Corollary \ref{3.cor:minimizers}, for a.e.~$x_0'\in\omega$ we have 
\begin{align*}
\lim_{m\to\infty} & \min_{\substack{\vect w\in H_0^1(D,\R^2), \\ v\in H_0^2(D)}}
 \int_{D} Q(x_0' + s_mx', \mat M_1 + \sym\nabla'\vect w,\mat M_2 - \nabla'^2v)\dd x' \\ 
& = \min_{\substack{\vect w\in H_0^1(D,\R^2), \\ v\in H_0^2(D)}}
\int_{D} Q(x_0', \mat M_1 + \sym\nabla'\vect w,\mat M_2 - \nabla'^2v)\dd x'
=|D|Q(x_0',\mat M_1, \mat M_2)\,.
\end{align*}
The latter equality follows by the convexity of $Q$. The obtained chain of inequalities 
\begin{align*}
\limsup_{m\to\infty}\limsup_{n\to\infty}\widetilde{\mathcal{K}}_{\chi^{h_{n,m}}}^0(\mat M_1 + x_3\mat M_2, D, B(r,r)) 
&\leq |D|Q(x_0',\mat M_1, \mat M_2) \\ 
& \leq \liminf_{m\to\infty}\liminf_{n\to\infty}\widetilde{\mathcal{K}}_{\chi^{h_{n,m}}}^0(\mat M_1 + x_3\mat M_2, D, B(r,r))
\end{align*}
implies the claim.
\end{proof}

\begin{proof}[Proof of Theorem \ref{thm:2}]

\noindent\fbox{\em(i) $\Rightarrow$ (ii)} 
We will prove that for a.e.~$x_0'\in\omega$, $Q(x_0',\cdot,\cdot)\in \set G_{\theta(x_0')}$, and the claim will follow by 
Proposition \ref{4.prop:geqp}. Define $(\chi^{h_{n,m}})_{n,m}$ as before, then for every $i=1,\ldots,N$ and
$\phi\in L^1([0,1]^2\times I)$, we calculate for a.e. $x_0' \in \omega$
\begin{align}
\lim_{m\to\infty}& \lim_{n\to\infty} \left|\int_{[0,1]^2\times I} \left(\chi_i^{h_{n,m}}(x) 
- \mu_i(x_0',x_3)\right)\phi(x)\dd x \right| \nonumber\\
& = \lim_{m\to\infty}\lim_{n\to\infty}
\frac{1}{s_m^2} \left|\int_{(x_0'+s_m[0,1]^2)\times I} 
\left(\chi_i^{h_n}(x) - \mu_i(x_0',x_3)\right)\phi\Big(\frac{x'-x_0'}{s_m}, x_3\Big)\dd x \right|\nonumber\\
& = \lim_{m\to\infty}
\frac{1}{s_m^2} \left|\int_{(x_0'+s_m[0,1]^2)\times I} 
\big(\mu_i(x',x_3) - \mu_i(x_0',x_3)\big)\phi\Big(\frac{x'-x_0'}{s_m},x_3\Big)\dd x \right|\nonumber \\
& \leq \|\phi\|_{L^\infty} \lim_{m\to\infty}\frac{1}{s_m^2} \int_{(x_0'+s_m[0,1]^2)\times I} 
\left|\mu_i(x',x_3) - \mu_i(x_0',x_3)\right|\dd x = 0\,. \label{4.eq:limlim_chi}
\end{align}
The last equality is easy to argument by the Lebesgue dominated convergence theorem and using the facts that $\mu_i \in L^{\infty} (\omega\times I)$ 
and for a.e.~$x_0' \in \omega$, $(x_0',x_3)$ is the Lebesgue point for $\mu_i(\cdot,x_3)$ for a.e. $x_3 \in I$. 
Next we proceed as in the proof of Proposition \ref{4.prop:closure}.
For countable dense families: $(D_k)_{k}$ of open Lipschitz subsets in $[0,1]^2$; 
$(\mat M_{1,j},\mat M_{2,j})_j\subset \R_{\sym}^{2\times 2}\times \R_{\sym}^{2\times 2}$;
$(\phi_k)_k\subset C([0,1]^2\times I)$, dense in $L^1([0,1]^2\times I)$; and monotonically decreasing to zero sequence $(r_l)_l$,
define doubly indexed family $(g_{n,m})_{n,m\in\N}$ by
\begin{align*}
g_{n,m} & = \sum_{k=1}^\infty 2^{-k}\min\left\{\max_{i=1,\ldots,N}\left|
\int_{[0,1]^2\times I} \left(\chi_i^{h_{n,m}}(x) 
- \mu_i(x_0',x_3)\right)\phi_k(x)\dd x \right|\,, 1\right\}\\
& \quad + \sum_{j,k,l=1}^\infty 2^{-j-k-l}\min\left\{\left|\frac{1}{|D_k|}
\widetilde{\mathcal{K}}_{\chi^{h_{n,m}}}^0(\mat M_{1,j} + x_3\mat M_{2,j}, D_k, B(r_l,r_l)) - 
Q(x_0',\mat M_{1,j},\mat M_{2,j})\right|\,, 1\right\}\,.
\end{align*}
Employing (\ref{4.eq:limlim_chi}) and the previous lemma, it follows that 
$\limsup_{m\to\infty}\limsup_{n\to\infty}g_{n,m} = 0$. By already utilized diagonal procedure, there exists an increasing
function $n\mapsto m(n)$ such that $\lim_{n\to\infty}g_{n,m(n)} = 0$.
The latter implies 
\begin{align*}
&\lim_{n\to\infty}\int_{[0,1]^2\times I} \left(\chi_i^{h_{n,m(n)}}(x) 
- \mu_i(x_0',x_3)\right)\phi_k(x)\dd x = 0\,, \quad \forall k\in\N\,,
\ i = 1,\ldots,N\,,\\
&\lim_{n\to\infty}\left(\frac{1}{|D_k|}
\widetilde{\mathcal{K}}_{\chi^{h_{n,m(n)}}}^0(\mat M_{1,j} + x_3\mat M_{2,j}, D_k, B(r_l,r_l)) - 
Q(x_0',\mat M_{1,j},\mat M_{2,j})\right) = 0\,, \quad \forall j,k,l\in\N\,,
\end{align*}
and we conclude the proof as in Proposition \ref{4.prop:closure}.

\noindent\fbox{\em(i) $\Leftarrow$ (ii)} For readability reasons, the proof of this direction is devided 
into {\em Steps 1--4}.

\noindent{\em Step 1.} Performing completely analogous construction from \cite[Theorem 3.5]{BaBa09}, which utilizes the Lusin and
the Scorza Dragoni theorems (cf.~\cite{EkTe99}), as well as the diagonal procedure,
for every $m\in\N$ one constructs a sequence $(\chi^{n,m})_n\subset\set X^N(\Omega)$
satisfying:
\begin{enumerate}[(i)]
  \item 
\begin{equation}\label{4.eq:wstar_conv}
\lim_{m\to\infty}\lim_{n\to\infty}\int_{\omega\times I}\chi^{n,m}_i(x',x_3)\phi(x')\dd x = 
\int_{\omega}\theta_i(x')\phi(x')\dd x'
\end{equation}
for every $i=1,\ldots,N$ and $\phi\in L^1(\omega)$. 
\item for each $m\in\N$ there exists a partition of $\Omega$
into a finite number of $M(m)$ open Lipschitz subsets $\{U_{j,m}\times I\}_{j=1,\ldots,M(m)}$ such that
$|U_{1,m}\times I|\to 0$ as $m\to\infty$, and for every $j=2,\ldots,M(m)$, 
there exists $x_{j,m}'\in U_{j,m}$ such that the sequence 
$(\chi^{n,m}|_{U_{j,m}\times I})_n \subset \set X^N(U_{j,m}\times I)$ has the homogeneous limit energy density $Q(x_{j,m}',\cdot,\cdot)\in\set G_{\theta(x_{j,m}')}$.

\item for the sequence of quadratic functions defined by
\begin{equation*}
Q_m(x',\mat M_1,\mat M_2) = \mathbbm{1}_{U_{1,m}}Q(x',\mat M_1,\mat M_2) 
+ \sum_{j=2}^{M(m)}\mathbbm{1}_{U_{j,m}}Q(x'_{j,m},\mat M_1,\mat M_2)
\end{equation*}
one can prove 
\begin{equation}\label{4.eq:qm_bound}
\lim_{m\to\infty}\int_\omega \sup_{|\mat M_1| + |\mat M_2| \leq r}
|Q_m(x',\mat M_1,\mat M_2) - Q(x',\mat M_1,\mat M_2)|\dd x' = 0
\end{equation}
for every $r>0$. From this we can assume that $(Q_m)_m$ converges to $Q$, i.e.~for a.e.~$x'\in\omega$ and for all 
$\mat M_1,\mat M_2\in \R_{\sym}^{2\times2}$, 
$\lim_{m\to\infty}Q_m(x',\mat M_1,\mat M_2) = Q(x',\mat M_1,\mat M_2)$.
This follows by taking the subsequence such that we have pointwise convergence of the integral function in \eqref{4.eq:qm_bound}.
\end{enumerate}
The proof is done first by taking a sequence of compact subsets $K_m$ of $\Omega$ such that 
$|\Omega \backslash K_m| \leq \frac{1}{m}$ and such that functions $\theta_i$ and 
$Q (\cdot, \cdot, \cdot)$ are continuous on $K_m$. Then we divide $\omega$ into $k$ 
pieces $\{U_{j,k,m}\}_{j=1,\dots,k}$ with Lipschitz boundary such that 
$\lim_{k \to \infty}\max_{1 \leq j \leq k} \diam U_{j,k,m}=0$. Next, in every $U_{j,k,m}$ which
intersects with $K_m$, we choose $x'_{j,m} \in K_m \cap U_{j,k,m}$. Those which do not intersect 
with $K_m$ we join in $U_{1,k,m}$. For each $x'_{j,k,m}$ we choose a sequence 
$(\chi^{n,j,k,m}|_{U_{j,k,m}\times I})_n \subset \set X^N(U_{j,k,m}\times I)$ that has the homogeneous 
limit energy density $Q(x_{j,k,m}',\cdot,\cdot)\in\set G_{\theta(x_{j,k,m}')}$. 
Then we fill the domain with $N$ materials on each $U_{j,k,m}$ according to the sequence of 
characteristic functions (on $U_{1,k,m}$ we put, e.g.~the material with energy density $Q_1$). 
In this way we define the sequence $(\chi^{n,k,m})_n\subset \set X^N(\Omega)$. 
Finally, we perform the diagonalization to find $M(m)=k(m)$ for which \eqref{4.eq:wstar_conv} 
and \eqref{4.eq:qm_bound} is satisfied (see \cite{BaBa09} for details). 
 \\
\noindent{\em Step 2 (Upper bound).} For arbitrary sequence $(h_n)_n$, $h_n\downarrow0$, and for every $m\in\N$,
denote $\chi^{h_{n,m}}= \chi^{n,m}$.
For fixed $m\in\N$ and arbitrary $r>0$, we estimate 
\begin{align*}
\limsup_{n\to\infty} &\, \widetilde{\set K}_{\chi^{h_{n,m}}}^0(\mat M_1 + x_3\mat M_2, A, B(r)) 
\leq \widetilde{\set K}_{(\chi^{h_{n,m}})}(\mat M_1 + x_3\mat M_2, A)\\
& = \widetilde{\set K}_{(\chi^{h_{n,m}}|_{U_{1,m}\times I})}(\mat M_1 + x_3\mat M_2, A\cap U_{1,m}) 
+ \sum_{j=2}^{M(m)}\widetilde{\set K}_{(\chi^{h_{n,m}}|_{U_{j,m} \times I})}(\mat M_1 + x_3\mat M_2, A\cap U_{j,m})\\
& =\widetilde{ \set K}_{(\chi^{h_{n,m}}|_{U_{1,m} \times I})}(\mat M_1 + x_3\mat M_2, A\cap U_{1,m}) 
+ \sum_{j=2}^{M(m)}|A\cap U_{j,m}|Q(x_{j,m}',\mat M_1,\mat M_2)\\
& = \widetilde{\set K}_{(\chi^{h_{n,m}}|_{U_{1,m} \times I})}(\mat M_1 + x_3\mat M_2, A\cap U_{1,m}) 
+ \int_{A\backslash U_{1,m}} Q_m(x',\mat M_1,\mat M_2)\dd x'\,.
\end{align*}
Using the fact that $|U_{1,m}|\to 0$ as $m\to\infty$, identity (\ref{4.eq:qm_bound}) 
implies the following upper bound
\begin{equation}\label{4.eq:limlimsup_ub}
\limsup_{m\to\infty}\limsup_{n\to\infty}\widetilde{\set K}_{\chi^{h_{n,m}}}^0(\mat M_1 + x_3\mat M_2, A, B(r)) 
\leq \int_A Q(x',\mat M_1,\mat M_2)\dd x\,. 
\end{equation}
{\em Step 3 (Lower bound).} Again take an arbitrary decreasing to zero sequence $(h_n)_n$, and denote
$\chi^{h_{n,m}} = \chi^{n,m}$ with $\chi^{n,m}$ constructed in {\em Step 1}. For each $m\in\N$
take (on a subsequence of $(h_n)_n$) a minimizing sequence $(\vect\psi^{n,m})_n\subset H^1(A\times I,\R^3)$ 
such that $\vect\psi^{n,m} = 0$
on $\pa A\times I$, $\|(\psi^{n,m}_{1}, \psi^{n,m}_{2}, h_n\psi^{n,m}_{3})\|_{L^2} \leq r$ and 
\begin{equation*}
\widetilde{\set K}_{\chi^{h_{n,m}}}^0(\mat M_1 + x_3\mat M_2, A, B(r)) = \int_{A\times I}
Q^{\chi^{h_{n,m}}}(x,\imath(\mat M_1 + x_3\mat M_2) + \nabla_{h_n}\vect\psi^{n,m})\dd x
\end{equation*}
for all $n,m\in\N$. Using the equi-coercivity of $Q^{\chi^{h_{n,m}}}$ and Lemma \ref{app:lem.limsup}, there
exist $\vect w\in H_0^1(A,\R^2)$, $v\in H_0^2(A)$, and $(\tilde{\vect\psi}^{n,m})_n\subset H^1(A\times I,\R^3)$
satisfying $\|\vect w\|_{L^2}^2 + \|v\|_{L^2}^2 \leq r^2$, $\tilde{\vect\psi}^{n,m} = 0$ on $\pa A\times I$,
$(\tilde\psi^{n,m}_{1}, \tilde\psi^{n,m}_{2}, h_n\tilde\psi^{n,m}_{3})\to 0$ strongly in the $L^2$-norm as 
$n\to \infty$, and the following identity holds (again on a subsequence)
\begin{equation}\label{4.eq:symgrad_nm}
\sym\nabla_{h_n}\vect\psi^{n,m} = \imath(\sym\nabla'\vect w - x_3\nabla'^2v) + \sym\nabla_{h_n}\tilde{\vect\psi}^{n,m}\,.
\end{equation}
By the definition of $\tilde{\set{K}}_{(\chi^{h_{n,m}})}$ and using (\ref{4.eq:symgrad_nm}) it follows
\begin{align*}
\liminf_{n\to\infty}\widetilde{\set K}_{\chi^{h_{n,m}}}^0(\mat M_1 + x_3\mat M_2, A, B(r)) & \geq
\sum_{j=2}^{M(m)}\widetilde{\set K}_{(\chi^{h_{n,m}}|_{U_{j,m} \times I})}(\mat M_1 + x_3\mat M_2 + \sym\nabla'\vect w - x_3\nabla'^2v, A\cap U_{j,m})\\
& = \sum_{j=2}^{M(m)}\int_{A\cap U_{j,m}} Q(x_{j,m}',\mat M_1 + \sym\nabla'\vect w,\mat M_2 - \nabla'^2v)\dd x'\,.
\end{align*}
Since
\begin{equation*}
\lim_{m\to\infty}\int_{A\cap U_{1,m}} Q(x',\mat M_1 + \sym\nabla'\vect w,\mat M_2 - \nabla'^2v)\dd x' = 0\,,
\end{equation*}
by the Lebesgue dominated convergence theorem we conclude
\begin{equation*}
\liminf_{m\to\infty}\liminf_{n\to\infty}\widetilde{\set K}_{\chi^{h_{n,m}}}^0(\mat M_1 + x_3\mat M_2, A, B(r)) \geq
\int_{A} Q(x',\mat M_1 + \sym\nabla'\vect w,\mat M_2 - \nabla'^2v)\dd x'\,,
\end{equation*}
hence, the lower bound 
\begin{align}
\liminf_{m\to\infty}\liminf_{n\to\infty}\widetilde{\set K}_{(\chi^{n,m})}^0(\mat M_1 + x_3\mat M_2, A, B(r)) & \geq
\min_{\substack{\vect w\in H_0^1(A,\R^2),\, v\in H^2_0(A) \\ \|\vect w\|_{L^2}^2 + \|v\|_{L^2}^2 \leq r^2}}
\int_{A} Q(x',\mat M_1 + \sym\nabla'\vect w,\mat M_2 - \nabla'^2v)\dd x' \nonumber\\
& =: m_Q(\mat M_1,\mat M_2,A,B(r)) \label{4.eq:limliminf_lb}
\end{align}
is established.

\noindent{\em Step 4 (Diagonalization).} Likewise in the first part of the proof, take dense families: 
$(D_k)_{k}$ of open Lipschitz subsets in $\omega$; 
$(\mat M_{1,j},\mat M_{2,j})_j\subset \R_{\sym}^{2\times 2}\times \R_{\sym}^{2\times 2}$;
$(\phi_k)_k\subset L^1(\omega)$; and monotonically decreasing to zero sequence $(r_l)_l$,
and define doubly indexed family $(g_{n,m})_{n,m\in\N}$ by


\begin{align*}
g_{n,m} & = \sum_{k=1}^\infty 2^{-k}\min\left\{\max_{i=1,\ldots,N}\left|
\int_{\omega\times I}\left(\chi_i^{h_{n,m}}(x',x_3) - \theta_i(x')\right)\phi_k(x')\dd x\right|\,, 1\right\}\\
& \quad + \sum_{j,k,l=1}^\infty 2^{-j-k-l}\min\left\{\left(
\widetilde{\mathcal{K}}_{\chi^{h_{n,m}}}^0(\mat M_{1,j} + x_3\mat M_{2,j}, D_k, B(r_l)) - 
\int_{D_k}Q(x',\mat M_{1,j},\mat M_{2,j})\dd x'\right)_+\,, 1\right\}\\
& \quad + \sum_{j,k,l=1}^\infty 2^{-j-k-l}\min\left\{\left(
m_Q(\mat M_1, \mat M_2, D_k, B(r_l)) - 
\widetilde{\mathcal{K}}_{\chi^{h_{n,m}}}^0(\mat M_{1,j} + x_3\mat M_{2,j}, D_k, B(r_l))\right)_+\,, 1\right\}\,,
\end{align*}
where $(\,\cdot\,)_+ = \max\{\,\cdot\,,0\}$.
Invoking (\ref{4.eq:wstar_conv}), upper bound (\ref{4.eq:limlimsup_ub}) and lower bound (\ref{4.eq:limliminf_lb}),
we infer that \\ $\limsup_{m\to\infty}\limsup_{n\to\infty}g_{n,m} = 0$. Again using diagonal procedure, there exists
an increasing function $n\mapsto m(n)$, such that $\lim_{n\to\infty} g_{n,m(n)} = 0$, which implies 
\begin{align*}
&\lim_{n\to\infty}\int_{\omega\times I}\left(\chi_i^{h_{n,m(n)}}(x',x_3) - \theta_i(x')\right)\phi_k(x')\dd x 
= 0\,, \quad \forall k\in\N\,, \ i = 1,\ldots,N\,,\\
&\limsup_{n\to\infty}\widetilde{\mathcal{K}}_{\chi^{h_{n,m(n)}}}^0(\mat M_{1,j} + x_3\mat M_{2,j}, D_k, B(r_l)) \leq   
\int_{D_k}Q(x',\mat M_{1,j},\mat M_{2,j})\dd x'\,, \quad \forall j,k,l\in\N\,, \\
& \liminf_{n\to\infty}\widetilde{\mathcal{K}}_{\chi^{h_{n,m(n)}}}^0(\mat M_{1,j} + x_3\mat M_{2,j}, D_k, B(r_l)) \geq
m_Q(\mat M_1, \mat M_2, D_k, B(r_l))\,, \quad \forall j,k,l\in\N\,.
\end{align*}
First, by the density argument observe that the sequence $\tilde\chi^{h_n'}:= \chi^{h_{n,m(n)}}$, $n\in\N$, satisfies that $\int_I \tilde\chi^{h_n'}(\cdot,x_3)\dd x_3$ 
weakly* converges to $\theta$ in $L^\infty(\omega,[0,1]^N)$.
After noticing that
\begin{equation*}
\lim_{l\to\infty}m_Q(\mat M_{1,j},\mat M_{2,j}, D_k, B(r_l)) = 
\int_{D_k} Q(x',\mat M_{1,j},\mat M_{2,j})\dd x'\,,
\quad \forall j,k\in\N\,,
\end{equation*}
again by the density argument we infer
\begin{align*}
\widetilde{\set K}_{(\tilde\chi^{h_n'})}(\mat M_1 + x_3\mat M_2, A) = \int_A Q(x',\mat M_1,\mat M_2)\dd x'\,
\end{align*} 
for all $\mat M_1, \mat M_2\in\R_{\sym}^{2\times 2}$ and $A\subset \omega$ open subset with Lipschitz boundary. 
Finally, taking a suitable subsequence, still denoted by $(h_n')_n$, such that $\tilde{\chi}^{h_n'}$ weakly* converges, 
finishes the proof.
\end{proof}

\section*{Appendix}

\setcounter{theorem}{0}
\renewcommand{\thesection}{A} 

\begin{theorem}[Korn inequalities, \cite{Hor95}]\label{app:thm.korn}
Let $p>1$, $\Omega\subset\R^3$ and $\Gamma\subset\pa\Omega$ of positive measure, then the following inequalities hold:
\begin{align}
\|\mat{\psi}\|_{W^{1,p}}^p &\leq C_K\left(\|\mat{\psi}\|_{L^p}^p + \|\sym\nabla \mat\psi\|_{L^p}^p\right)\,,
\quad \forall \mat{\psi}\in W^{1,p}(\Omega,\R^3)\,,\\ \label{Kornwithbc}
\|\mat{\psi}\|_{W^{1,p}}^p &\leq C_K^{\Gamma}(\|\vect\psi\|^p_{L^p(\Gamma)}+\left\|\sym\nabla \mat\psi\|_{L^p}^p\right)\,,
\quad \forall \mat{\psi}\in W^{1,p}(\Omega,\R^3)\,,
\end{align}
where positive constants $C_K$ and $C_K^{\Gamma}$ depend only on $p$, $\Omega$ and $\Gamma$.
\end{theorem}

\begin{theorem}[Griso's decomposition, \cite{Gri05}]\label{aux:thm.griso}
Let $ \omega\subset\R^2$ with Lipschitz boundary and $\mat{\psi}\in H^1(\omega\times I,\R^3)$, then for arbitrary $h>0$ 
the following identity holds
\begin{equation}\label{griso1}
\mat{\psi} = \hat{\mat{\psi}}(x') + \vect{r}(x')\wedge x_3\vect{e}_3 + \bar{\mat{\psi}}(x)
= \left\{   \begin{array}{l}
				\hat{\psi}_1(x') + r_2(x')x_3 + \bar{\psi}_1(x)\\
				\hat{\psi}_2(x') - r_1(x')x_3 + \bar{\psi}_2(x)\\
				\hat{\psi}_3(x') + \bar{\psi}_3(x)
			\end{array}
  \right.\,,
\end{equation}
where
\begin{equation} \label{griso2}
\hat{\mat\psi}(x') = \int_I \mat\psi(x',x_3)\dd x_3\,,\quad \vect{r}(x') 
= \frac32\int_Ix_3\vect{e}_3\wedge\mat\psi(x',x_3)\dd x_3\,.
\end{equation}
Moreover, the following inequality holds
\begin{equation}\label{prvaKorn}
\|\sym\nabla_h(\hat{\mat\psi} + \vect{r}\wedge x_3\vect{e}_3)\|_{L^2}^2 
+ \|\nabla_h\bar{\mat\psi}\|_{L^2}^2 + \frac{1}{h^2}\|\bar{\mat\psi}^h\|_{L^2}^2
\leq C\|\sym\nabla_h\mat\psi\|_{L^2}^2\,,
\end{equation}
with constant $C>0$ depending only on $\omega$.
\end{theorem}
The following corollary is the direct consequence of Theorem \ref{aux:thm.griso} (relation \eqref{prvaKorn}, i.e. \eqref{3.eq:compact_est}), Poincare and Korn inequalities.
\begin{corollary}[Korn's inequality for thin domains]\label{kornthincor}
Let $\omega \subset \R^2$ and $\Gamma \subset \partial \Omega$ of positive measure, then the following inequalities hold
\begin{align*}
	\|(\psi_1,\psi_2, h_n\psi_3)\|^2_{H^1} &\leq C_T\left(\|(\psi_1,\psi_2, h_n\psi_3)\|_{L^2}^2 + \|\sym\nabla_h \mat\psi\|_{L^2}^2\right)\,,
	\quad \forall \mat{\psi}\in H^1(\omega \times I,\R^3);\,\\
	\|(\psi_1,\psi_2, h_n\psi_3)\|_{H^1}^2 &\leq C_T^{\Gamma}(\||(\psi_1,\psi_2, h_n\psi_3)\|^2_{L^2(\Gamma)}+\left\|\sym\nabla_h \mat\psi\|_{L^2}^2\right)\,,
	\quad \forall \mat{\psi}\in H^1(\omega \times I,\R^3)\,,
\end{align*}
where positive constants $C_T$ and $C_T^{\Gamma}$ depend only on  $\omega$ and $\Gamma$.
\end{corollary}
The following lemma tells us additional information on the weak limit of sequence that has bounded symmetrized scaled gradients. 
\begin{lemma}\label{app:lem.limsup}
Let $\omega\subset \R^2$ be a bounded set  with Lipschitz boundary and $(h_n)_n$ monotonically decreasing
to zero sequence of positive reals. Let $(\mat{\psi}^{h_n})_{n}\subset H^1(\omega\times I,\R^3)$  which for all $n\in\N$
equals zero on $\Gamma_d\times I\subset\pa\omega\times I$ of strictly positive surface measure,  and
\begin{equation*}
\limsup_{n\to\infty}\|\sym \nabla_{h_n}\mat{\psi}^{h_n}\|_{L^2} < \infty\,,
\end{equation*}
then there exists a subsequence (still denoted by $(h_n)_n$) for which it holds
\begin{equation*}
\sym\nabla_{h_n}\mat\psi^{h_n} = \imath(-x_3\nabla'^2 v + \sym\nabla' \vect{w}) + \sym\nabla_{h_n}\bar{\mat\psi}^{h_n},
\end{equation*}
for some 
$v\in H_{\Gamma_d}^2(\omega)$, $\vect{w}\in H_{\Gamma_d}^1(\omega,\R^2)$ and a sequence $(\bar{\mat\psi}^{h_n})_{n}\subset H^1(\omega\times I,\R^3)$
satisfies $\bar{\mat\psi}^{h_n} = 0$ on $\Gamma_d\times I$ and $(\bar{\psi}_1^{h_n},\bar\psi_2^{h_n},h_n\bar\psi_3^{h_n})\to 0$ 
in the $L^2$-norm. Furthermore,
\begin{equation*}
\|v\|_{L^2}^2 + \|\vect{w}\|_{L^2}^2 \leq \limsup_{n\to\infty}\|(\psi_1^{h_n},\psi_2^{h_n},h_n\psi_3^{h_n})\|_{L^2}^2\,.
\end{equation*}
\end{lemma}
\begin{proof}
Applying the Griso's decomopsition (Theorem \ref{aux:thm.griso}) on each $\vect \psi^{h_n}$, we find
\begin{equation*}
\vect \psi^{h_n}(x) = \hat{\vect{\psi}}^{h_n}(x') + \vect r^{h_n}(x')\wedge x_3\vect e_3 
+ \bar{\vect{u}}^{h_n}(x)\,,
\end{equation*}
with sequences 
$\hat{\vect{\psi}}^{h_n}(x') = \int_I \vect{\psi}^{h_n}(x',x_3)\dd x_3$,
$\vect{r}^{h_n}(x')= \frac32\int_Ix_3\vect{e}_3\wedge\vect{\psi}^{h_n}(x',x_3)\dd x_3$ and $\bar{\vect{\psi}}^{h_n}$
satisfying the following a priori estimate
\begin{align}
&\|\sym\nabla'(\hat{\psi}_1^{h_n},\hat \psi_2^{h_n})\|_{L^2}^2 + \frac{1}{12}\|\sym\nabla'(r_2^{h_n},-r_1^{h_n})\|_{L^2}^2
+ \frac{1}{2h_n^2}\|\pa_1(h_n\hat{\psi}_3^{h_n}) + r_2^{h_n}\|_{L^2}^2 \label{3.eq:compact_est}\\ \nonumber
&+ \frac{1}{2h_n^2}\|\pa_2(h_n\hat{\psi}_3^{h_n}) - r_1^{h_n}\|_{L^2}^2\,
+ \|\nabla_{h_n}\bar{\vect{\psi}}^{h_n}\|_{L^2}^2 + \frac{1}{h_n^2}\|\bar{\vect{\psi}}^{h_n}\|_{L^2}^2
\leq C\|\sym\nabla_{h_n}\vect{\psi}^{h_n}\|_{L^2}^2 \leq C\,.
\end{align}
Together with the Korn's inequality with boundary condition (cf.~Theorem \ref{app:thm.korn}), 
the above estimate implies (on a subsequence)
\begin{align*}
(\hat{\psi}_1^{h_n},\hat \psi_2^{h_n}) \rightharpoonup \vect w\quad\text{weakly in }H^1(\omega,\R^2)\quad \text{and}
\quad \vect r^{h_n} \rightharpoonup \vect r\quad \text{weakly in }H^1(\omega,\R^2)\,.
\end{align*}
Furthermore, using the triangle inequality and estimate (\ref{3.eq:compact_est}) 
\begin{equation*}
\pa_1(h_n\hat{\vect{\psi}}^{h_n}_3) \to -r_2\quad \text{and}\quad \pa_2(h_n\hat{\vect{\psi}}^{h_n}_3) \to r_1 
\quad\text{strongly in }L^2(\omega)\,.
\end{equation*}
By the compactness of the trace operator, $\vect r\in H^1_{\Gamma_d}(\omega,\R^2)$, thus, by the Korn's inequality, there 
exists $v\in H^2_{\Gamma_d}(\omega)$ such that
\begin{equation*}
h_n\hat{\vect{\psi}}^{h_n}_3 \to v \quad\text{ strongly in }H^1(\omega)\,,\quad \text{and}\quad 
r_1 = \pa_2 v\,,\quad r_2 = -\pa_1v\,.
\end{equation*}
Defining 
\begin{equation*}
{\bar{\vect{\psi}}}^{h_n}(x) = \vect{\psi}^{h_n}(x) - \left(\begin{array}{c}\vect{w}(x') \\ \dfrac{v(x')}{h_n}\end{array}\right) 
+ x_3\left(\begin{array}{c} \nabla'v(x') \\ 0 \end{array} \right)\,,
\end{equation*}
it is easy to check $(\bar\psi_1^{h_n},\bar\psi_2^{h_n},h_n\bar\psi_3^{h_n})\to 0$ in the $L^2$-norm, which finishes the proof. 
\end{proof} 
The following lemma is given in \cite[Proposition 3.3]{Vel14a}. It is a consequence of Theorem \ref{aux:thm.griso}. It tells us how we can further decompose the sequence of deformations that has bounded symmetrized scaled gradient. 
\begin{lemma}\label{igor2}
	Let $ A \subset \R^2$ with $C^{1,1}$ boundary. Denote by $\{A_i\}_{i=1,\dots,k}$ the connected components of $A$.
		Let $(\vect\psi^{h_n})_n \subset H^1(A \times I,\R^3)$ be such that
		\begin{eqnarray*}
			\label{eq:pp1}&& (\psi_1^{h_n},\psi_2^{h_n},h_n\psi_3^{{h_n}}) \to 0, \text{ strongly in } L^2,\quad \forall n \in \N,\ \forall i=1, \dots,k \,,\ \int_{A_i} \psi_3^{h_n}=0\,, \\ 
			&& \limsup_{n \to \infty} \|\sym \nabla_{h_n} \vect \psi^{h_n}\|_{L^2} \leq M<\infty\,.
		\end{eqnarray*}
		Then there exist $(\varphi^{h_n})_{n \in \N} \subset H^2(A)$, $(\tilde{\vect \psi}^{h_n})_{n \in \N} \subset H^1(A \times I ,\R^3)$
		such that
		$$ \sym \nabla_{h_n}\vect \psi^{h_n}=-x_3\iota(\sym \nabla'^2 \varphi^{h_n})+\sym \nabla_{h_n} \tilde{\vect \psi}^{h_n}+o^{h_n},$$
		where $o^{h_n} \in L^2(A \times I,\R^{3 \times 3})$ is such that $o^{h_n} \to 0$, strongly in $L^2$, and the following properties hold
		\begin{eqnarray*}
			&&\label{igor10} \lim_{n \to \infty} \left(\|\varphi^{h_n}\|_{H^1}+\|\tilde{\vect \psi}^{h_n}\|_{L^2} \right)=0\,,\\
			&& \label{igor11} \limsup_{n \to \infty} \left( \| \varphi^{h_n} \|_{H^2}+\|\nabla_{h_n} \tilde{\vect \psi}^{h_n}\|_{L^2} \right)\leq C(A) M\,.
		\end{eqnarray*}

\end{lemma}

The following lemma is given in \cite[Lemma 3.6]{Vel14a}. Although the claim can be put in more general form (since the argument is simple truncation), we state it only in the form we need here. 
\begin{lemma} \label{nulizacija}
	Let $p\geq 1$ and $A\subset \R^2$ be an open, bounded set.
Let $(\vect\psi^{h_n})_{n \in \N} \subset W^{1,p}(A \times I,\R^3)$ and $(\varphi^{h_n})_{n \in \N} \subset W^{2,p}(A)$. Suppose that
	$\left(|\nabla'^2 \varphi^{h_n}|^p\right)_{n\in\N}$ and $\left(| \nabla_{h_n} \vect\psi^{h_n} |^p\right)_{n \in \N}$ are equi-integrable and
	\begin{equation*}
	\lim_{n \to\infty} \left(  \| \varphi^{h_n} \|_{W^{1,p}}+\|\vect \psi^{h_n} \|_{L^p} \right)=0.
	\end{equation*}
	Then there exist sequences $(\tilde{\varphi}^{h_n})_{n \in \N} \subset W^{2,p}(A)$, $(\tilde{\psi}^{h_n})_{n\in \N}\subset W^{1,p}(A \times I,\R^3)$ and a sequence of sets $(A_n)_{n\in \N}$  such that for each $n \in\N$, $A_n \ll A_{n+1} \ll A$ and $\cup_{n \in \N} A_n=A$ and
	\begin{enumerate}[(a)]
		\item $\tilde{\varphi}^{h_n}=0$, $\nabla' \tilde{\varphi}^{h_n}=0$  in a neighborhood of $\partial A$,  $\tilde{\vect \psi}^{h_n}=0$ in a neighborhood of $\partial A \times I$;
		\item $\tilde{\vect\psi}^{h_n}=\vect\psi^{h_n} \text{ on } A_n \times I,\ \tilde{\varphi}^{h_n}=\varphi^{h_n} \text{ on } A_n$;
		\item $ \|\tilde{\varphi}^{h_n}-\varphi^{h_n}\|_{W^{2,p}}\to0,\  \|\tilde{\vect\psi}^{h_n}-\vect\psi^{h_n}\|_{W^{1,p}}\to 0$, $\|\nabla_{h_n} \tilde{\vect\psi}^{h_n}-\nabla_{h_n}\vect \psi^{h_n}\|_{L^p} \to 0$, as $n \to \infty$.
	\end{enumerate}
\end{lemma}
The following two lemmas are given in \cite{FoMuPe} and \cite{BraZep}. The proof of the second claim in the first lemma is given in \cite[Proposition A.5]{Vel14a} as an adaptation
of the result given in \cite{FoMuPe}. 
\begin{lemma} \label{ekvi1}
	Let $p>1$ and $A \subset \R^n$ be an open bounded set.
	\begin{enumerate}[(a)]                                                                    \item Let $( w^n)_{n \in \N}$  be a bounded sequence in $W^{1,p}(A)$. Then there exist a subsequence $( w^{n(k)})_{k \in \N}$
		and a sequence $( z^k)_{k \in \N} \subset W^{1,p}(A)$ such that
		\begin{equation*}\label{stefan1} |\{ z^k \neq w^{n(k)} \}| \to 0  ,
		\end{equation*}
		as $k \to \infty$ and $\big(|\nabla  z_k|^p \big)_{k \in \N}$ is equi-integrable. Each $z^k$ may be chosen to be
		Lipschitz function. If $ w^n \rightharpoonup  w$ weakly in $W^{1,p}$, then $z^k \rightharpoonup  w$ weakly in $W^{1,p}$.
		\item Let $( w^n)_{n \in \N}$  be a bounded sequence in $W^{2,p}(A)$. Then there exist a subsequence $( w^{n(k)})_{k \in \N}$
		and a sequence $( z^k)_{k \in \N} \subset W^{2,p}(A)$ such that
		\begin{equation*} \label{stefan2}
			|\{ z^k \neq  w^{n(k)} \}|                                                                                                                                                                                                        \to 0  ,
		\end{equation*}                                                                     as $k \to \infty$ and $\big(|\nabla^2  z^k|^p \big)_{k \in \N}$ is equi-integrable. Each  $ z^k$ may be chosen such
		that $ z^k \in W^{2,\infty}(S)$. If $ w^n \rightharpoonup  w$ weakly in $W^{2,p}$, then $ z^k \rightharpoonup  w$ weakly in $W^{2,p}$.
	\end{enumerate}
\end{lemma}
\begin{lemma}\label{ekvi2}
 Let $p>1$ and $A \subset \R^2$ be an open bounded set with Lipschitz boundary. Let $(h_n)_{n\in \N}$ be a sequence of positive numbers converging to $0$ and let $(\vect w^{h_n})_{n \in\N} \subset W^{1,p} (A \times I, \R^3)$ be a bounded sequence   satisfying:
$$ \limsup_{n \in \N} \|\nabla_{h_n} \vect w^{h_n}\|_{L^p}<+\infty.$$ Then there exists a subsequence
$(\vect w^{h_{n(k)}})_{k \in \N}$ and a sequence $(\vect z^{h_{n(k)}})_{k \in \N}$ such that
$$  |\vect z^{h_{n(k)}} \neq \vect w^{h_{n(k)}} | \to 0,$$
as $k \to \infty$ and $|\nabla_{h_{n(k)}} \vect z^{h_{n(k)}}|^p$ is equi-integable. If $\vect w^{h_{n}}\rightharpoonup \vect w \in W^{1,p}(A,\R^3)$ weakly in $W^{1,p}$, then $\vect z^{h_{n(k)}}\rightharpoonup \vect w $ in $W^{1,p}$. Each $\vect z^{h_{n(k)}}$ may be chosen such that 
$|\nabla_{h_{n(k)}}\vect z^{h_{n(k)}}|$ is bounded. 
\end{lemma}
The following lemma is given in \cite[Lemma 3.4]{Vel14a}. For the sake of completeness, we will give also the proof. 
\begin{lemma}[continuity in $\mat M$] \label{lem:ocjena}
	Under the uniform coerciveness and boundedness assumption \eqref{2.eq:A1}, 
	there exists a constant $C>0$ depending only on $\alpha$ and $\beta$, such that for every sequence 
	$(h_n)_{n}$ monotonically decreasing to zero and $A \subset \omega$ open set with 
	Lipschitz boundary it holds:
	\begin{eqnarray}\label{ocjena1111}
	\left| \Khpm(\mat M_1,A)-\Khpm(\mat M_2,A)\right| 
	&\leq& C \|\mat M_1-\mat M_2\|_{L^2}\left(\|\mat M_1\|_{L^2}+\|\mat M_2\|_{L^2}\right)\,,\\
	\nonumber & &  \ \forall \mat M_1, \mat M_2 \in \set S(\omega)\,,
	\end{eqnarray}
\end{lemma}
\begin{proof}
	Due to relation \eqref{trivijala}, it is enough to assume $A=\omega$.
	For fixed $\mat M_1,\, \mat M_2 \in \set S(\omega)$ and $r,\,h_n>0$ take arbitrary
	 $\vect \psi_{\mat M_\alpha}^{r,h_n}\in H^1(\Omega,\R^3)$, which for $\alpha=1,2$ satisfy: 
	\begin{eqnarray}\label{definf}
	& &  \int_\Omega Q^{h_n}\left(x,\iota(\mat M_\alpha )+\nabla_{h_n} \vect \psi_{\mat M_\alpha}^{r,h_n}\right)\dd x 
	\leq \\ \nonumber & & \hspace{10ex}\inf_{\vect \psi \in H^1(\Omega,\R^3) \atop \|(\psi_1,\psi_2,h_n\psi_3)\|_{L^2} \leq r}
	\int_{\Omega} Q^{h_n}\left(x,\iota(\mat M_\alpha )+\nabla_{h_n} \vect \psi\right) \dd x + h_n\,; \\
	\nonumber& & \|(\psi_{\alpha,1}^{r,h_n},\psi_{\alpha,2}^{r,h_n},h_n\psi_{\alpha,3}^{r,h_n})\|_{L^2} \leq r\,.
	\end{eqnarray}
	We want to prove that for every $r>0$ we have
	\begin{eqnarray}\label{nped}
	& &\left| \int_{\Omega} Q^{h_n}\left(x,\iota(\mat M_1 )+\nabla_{h_n}\vect \psi_{\mat M_1 }^{r,h_n}\right)\dd x - 
	\int_{\Omega} Q^{h_n}\left(x,\iota(\mat M_2 )+\nabla_{h_n} \vect \psi_{\mat M_2}^{r,h_n}\right)\dd x \right| \\ 
	& & \nonumber \hspace{+15ex}\leq C \|\mat M_1-\mat M_2\|_{L^2}\left(\|\mat M_1\|_{L^2}+\|\mat M_2\|_{L^2}\right) + h_n\,.
	\end{eqnarray}
	From that, (\ref{ocjena1111}) can be easily obtained by using \eqref{bukal111111} for a family of balls of radius $r$.
	
	Let us prove (\ref{nped}). From (\ref{definf}) and \eqref{2.eq:A1}, by testing with  zero function, 
	we can assume for $\alpha=1,2$
	$$
	\alpha \|\mat M_\alpha+ \sym \nabla_{h_n} \vect \psi^{r,h_n}_{\mat M_\alpha}\|^2_{L^2}  
	\leq \int_{A \times I} Q^{h_n}\left(x,\iota(\mat M_\alpha)+\nabla_{h_n} \vect \psi_{\mat M_\alpha}^{r,h_n}\right)\dd x 
	\leq \beta \|\mat M_{\alpha}\|^2_{L^2}.
	$$
	From this we have for $\alpha=1,2$
	\begin{equation} \label{identitet} 
	\| \sym \nabla_{h_n} \vect \psi^{r,h_n}_{\mat M_\alpha}\|^2_{L^2} 
	\leq C(\alpha,\beta) \|\mat M_{\alpha}\|^2_{L^2}.
	\end{equation}
	
	Without any loss of generality we can also assume that
	\begin{equation}\label{pretpostavka1}
	\int_{\Omega} Q^{h_n} \left(x,\iota(\mat M_1)+\nabla_{h_n} \vect \psi_{\mat M_1}^{r,h_n}\right)\dd x 
	\geq \int_{\Omega} Q^{h_n}\left(x,\iota(\mat M_2)+\nabla_{h_n} \vect \psi_{\mat M_2}^{r,h}\right)\dd x\,.
	\end{equation}
	We have
	\begin{eqnarray*}
		& &\left| \int_{\Omega} Q^{h_n}\left(x,\iota(\mat M_1)+\nabla_{h_n} \vect \psi_{\mat M_1}^{r,h_n}\right)\dd x 
		- \int_{\Omega} Q^{h_n}\left(x,\iota(\mat M_2)+\nabla_{h_n} \vect \psi_{\mat M_2}^{r,h_n}\right)\dd x\right| \\ 
		& & = \int_{\Omega} Q^{h_n}\left(x,\iota(\mat M_1)+\nabla_{h_n} \vect \psi_{\mat M_1}^{r,h_n}\right)\dd x 
		- \int_{\Omega} Q^{h_n}\left(x,\iota(\mat M_2)+\nabla_{h_n} \vect \psi_{\mat M_2}^{r,h_n}\right)\dd x \\ & & =
		\int_{\Omega} Q^{h_n}\left(x,\iota(\mat M_1)+\nabla_{h_n} \vect \psi_{\mat M_1}^{r,{h_n}}\right)\dd x 
		- \int_{\Omega} Q^{h_n}\left(x,\iota(\mat M_1)+\nabla_{h_n} \vect \psi_{2}^{r,h_n}\right)\dd x \\ 
		& & + \int_{\Omega} Q^{h_n}\left(x,\iota(\mat M_1)+\nabla_{h_n} \vect \psi_{\mat M_2}^{r,h_n}\right)\dd x 
		- \int_{\Omega} Q^{h_n}\left(x,\iota(\mat M_2)+\nabla_{h_n} \vect \psi_{\mat M_2}^{r,h_n}\right)\dd x \\
		& & \leq h_n + C(\alpha,\beta) \|\mat M_1-\mat M_2\|_{L^2}\left(\|\mat M_1\|_{L^2}+\|\mat M_2\|_{L^2}\right),
	\end{eqnarray*}
	where we used \eqref{pretpostavka1}, \eqref{razlika} and \eqref{identitet} respectively. This concludes the proof. 
\end{proof} 

\section*{Acknowledgment}
This work has been supported by the Croatian Science
Foundation under Grant agreement No.~9477 (MAMPITCoStruFl).


\end{document}